\documentclass{article}
\usepackage{graphicx} 
\usepackage{comment} 
\usepackage{xcolor}
\usepackage{amsthm}

\usepackage{mathtools}
\usepackage{booktabs}
\usepackage{pifont}

\usepackage{todonotes}
\usepackage[normalem]{ulem}

\title{Counting Spinal Tree-Child Networks via Word Encodings and Generating Functions}

\author{Pau Vives$^{1}$, Anna de Mier$^{2}$, Gabriel Cardona$^{1}$, Joan Carles Pons$^{1,*}$}

\date{\today}
\usepackage{tikz}
\usepackage{amssymb}

\usepackage{amsmath}
\usepackage{tkz-graph}
\usetikzlibrary{shapes,positioning,calc}
\tikzstyle{mi}=[circle,draw,inner sep=0.25pt,minimum size=1.5mm]
\tikzstyle{x}=[circle,draw,inner sep=0.25pt,minimum size=1.5mm]
\tikzstyle{bx}=[rectangle,draw,inner sep=0.25pt,line width=0.2pt,minimum size=1.75mm]
\tikzstyle{bxl}=[rectangle,draw,inner sep=0.25pt,line width=0.2pt,minimum size=2.5mm]
\tikzstyle{label}=[rectangle,inner sep=-0.25pt, fill = white]
\tikzstyle{label2}=[circle,inner sep=-2pt,mimum size=0.1mm, fill = white]
\tikzstyle{ppalpath}=[style={->},
decoration={snake,pre length=0.05cm, post length=0.05cm,
amplitude=.4mm,segment length=2mm},
decorate]
\tikzset{
    between/.style args={#1 and #2}{
         at = ($(#1)!0.5!(#2)$)
    }
}
\newtheorem{thm}{Theorem} 

\newtheorem{prop}[thm]{Proposition} 
\newtheorem{lem}[thm]{Lemma} 
\newtheorem{coro}[thm]{Corollary} 
\newtheorem{conj}[thm]{Conjecture}

\tikzset{
    amunt/.style args={#1 and #2}{
         at = ($(#1)!0.25!(#2)$)
    }
}

\newcommand{\STC}{\mathcal{STC}}
\newcommand{\SCTC}{\mathcal{SCTC}}
\newcommand{\NLSCTC}{\mathcal{NLSCTC}}
\newcommand{\TC}{\mathcal{TC}}
\newcommand{\MT}{\mathcal{MT}}
\newcommand{\NLMT}{\mathcal{NLMT}}
\newcommand{\NLSTC}{\mathcal{NLSTC}}
\newcommand{\STCch}{\mathcal{STC}^{\text{ch}}}
\newcommand{\STCnc}{\mathcal{STC}^{\text{nc}}}
\newcommand{\NLSTCch}{\mathcal{NLSTC}^{\text{ch}}}
\newcommand{\NLSTCnc}{\mathcal{NLSTC}^{\text{nc}}}

\newcommand{\cA}{\mathcal{A}}
\newcommand{\cB}{\mathcal{B}}
\newcommand{\cC}{\mathcal{C}}
\newcommand{\cL}{\mathcal{L}}
\newcommand{\cR}{\mathcal{R}}

\DeclareMathOperator{\Seq}{Seq}
\newcommand{\boxprod}{\mathop{\prescript{\square}{}{\star}}}

\begin{document}

\maketitle

\begin{center}
\vspace{-8pt}

$^{1}$Department of Mathematics and Computer Science\\
University of the Balearic Islands, Palma 07122, Spain

\vspace{6pt}

$^{2}$Departament de Matem\`atiques and IMTech\\
Universitat Polit\`ecnica de Catalunya - Barcelona Tech (UPC)\\
  Barcelona 08034, Spain

\vspace{8pt}

$^{*}$Corresponding author: joancarles.pons@uib.es
\end{center}

\section*{Abstract}

We study the enumeration of spinal tree-child phylogenetic networks, a rigid family of tree-child networks in which all internal vertices lie on a single root--to--leaf path. We provide two complementary combinatorial frameworks. First, we introduce a word model: unlabeled spinal networks correspond to a suitable class of restricted words with fixed multiplicities, taken modulo a simple relabeling equivalence, which yields an explicit closed enumeration. Second, we develop a symbolic-method approach based on a marked version of trees that admits a clean recursive specification; its boxed-product translation leads to a solvable bivariate generating function and a direct derivation of the coefficients. 

\section{Introduction}

Phylogenetic networks extend phylogenetic trees by allowing reticulate evolutionary events such as hybridization or horizontal gene transfer \cite{morrison2005networks,huson2006application, huson2010phylogenetic}. From a combinatorial perspective, their enumeration is markedly more delicate than for trees: even in the rooted binary setting there are infinitely many networks on a fixed leaf set unless additional structural constraints are imposed. Consequently, the main counting questions are formulated within restricted classes and typically parametrised by the number of leaves $n$ and reticulation vertices $k$.

Exact and asymptotic counting results are known for several families (e.g.\ level-$1$ or $2$, and galled networks) via techniques including generating functions, recursive decompositions, and algorithmic enumeration (see  \cite{bouvel2020counting, gunawan2020counting} to name a few). Among the most studied classes are \emph{tree-child} networks, which satisfy a local condition ensuring that each internal vertex has at least one non-reticulation child \cite{cardona2008comparison}. Despite significant progress, obtaining a closed formula for $|\TC_{n,k}|$, the cardinality of the set of tree-child networks with \(n\) leaves and \(k\) reticulations, in full generality has been challenging \cite{fuchs2022counting,  fuchs2018counting, cardona2020counting, cardona2019generation, fuchs2020counting}.

In this direction, Pons and Batle (see \cite{pons2021combinatorial}) proposed a conjectural correspondence between tree-child networks and a family of restricted words, yielding a precise formula for $|\TC_{n,k}|$ in terms of word counts. Very recently, this conjecture appears to have been resolved \cite{lin2026proof} and prior work in this direction includes \cite{fuchs2021short,chang2024enumerative}. Motivated by this word-based viewpoint, we revisit a particularly rigid and tractable subclass: \emph{spinal tree-child} networks. These are tree-child networks admitting a \emph{spine}, i.e.\ a root-to-leaf path containing all internal vertices
\cite{francis2024phylogenetic}. Spinal networks have been previously enumerated using expanding covers, providing bijective encodings and explicit formulas for several spinal subclasses \cite{francis2025counting}.

The goal of this paper is twofold. First, we recover the enumeration of spinal tree-child networks through a word-encoding approach in the spirit of \cite{pons2021combinatorial}.
Second, we develop a symbolic-method approach: we give combinatorial specifications for spinal networks and derive generating functions that yield the same counting formulas.

The paper is organised as follows. Section~\ref{sec:Background} reviews the required background on phylogenetic networks and summarises the relevant previous work. Section~\ref{sec:words} establishes our word encodings for spinal tree-child networks and derives the enumeration of the unlabeled and labeled classes. Section~\ref{sec:gf} develops the symbolic specifications and generating-function derivations.

\section{Preliminaries}\label{sec:Background}

This section recalls the main definitions and notation concerning phylogenetic networks, with particular emphasis on tree-child and spinal networks. We also summarize the relevant previous work on word encodings and on the enumeration of spinal networks.

\subsection{Phylogenetic networks}

An \emph{unlabeled (binary) phylogenetic network} is a directed acyclic graph (DAG) \(N=(V,A)\) such that every vertex is exactly one of the following types:
\begin{itemize}
\item a \emph{root}, with indegree~0 and outdegree~1, and there is a single such node;
\item a \emph{leaf}, with indegree~1 and outdegree~0;
\item a \emph{tree vertex}, with indegree~1 and outdegree~2;
\item a \emph{reticulation vertex} (or simply \emph{reticulation}), with indegree~2 and outdegree~1.
\end{itemize}

A \emph{(binary) phylogenetic network} on a finite set \(X\) of \emph{taxa} is an unlabeled phylogenetic network together with a bijection between \(X\) and the set of leaves $L$. We will often identify the leaf set with $X$ and refer to the leaves as elements of $X$. Also, we will usually take the taxa set to be $X=[n]=\{1,2,\dots,n\}$, and we will refer to the leaf corresponding to $i$ as $i$.
Vertices in \(V \setminus L\)   are called \emph{internal vertices}. The unique outgoing arc of the root is referred to as the \emph{root arc}, and arcs directed into reticulation vertices are called \emph{reticulation arcs}.  
A phylogenetic network with no reticulation vertices is a \emph{tree}, denoted by \(T\).



Two networks \(N_1 = (V_1, A_1)\) and \(N_2 = (V_2, A_2)\) on the same leaf set \(X\) are \emph{isomorphic} if there exists a bijection \(\varphi: V_1 \to V_2\) such that $\varphi(x) = x$ for all $x\in X$ and $uv\in A_1$ if and only if $\varphi(u)\varphi(v)\in A_2$.

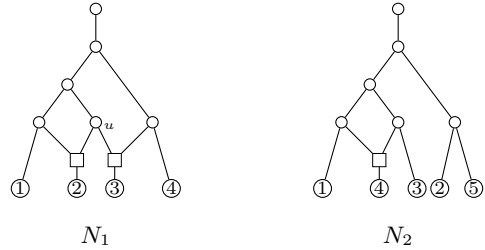
\begin{figure}[h]
\centering
    \begin{tikzpicture}[scale=0.5]
        \draw (0,0)                 node[mi]    (root1)      {};
        \draw (0,-1)                 node[mi]    (root)      {};
        \draw (-0.75,-2)                 node[mi]    (A)      {};
        \draw (-1.5,-3)                 node[mi]    (A1)      {};
        \draw (1.5,-3)                 node[mi]    (B)      {};
        \draw (0,-3)                 node[mi]    (A2)      {};
        \draw (-0.5,-4)                 node[bx]    (r1)      {};
        \draw (0.5,-4)                 node[bx]    (r2)      {};
        \draw (0.5,-4.75)                 node[mi]    (r12)      {\scriptsize $3$};
        \draw (-0.5,-4.75)                 node[mi]    (r22)      {\scriptsize $2$};
        \draw (-2,-4.75)                 node[mi]    (A11)      {\scriptsize $1$};
        \draw (2,-4.75)                 node[mi]    (B11)      {\scriptsize $4$};


        \draw[style={-}](root1)to(root);
        \draw[style={-}](A)to(root);
        \draw[style={-}](B)to(root);
        \draw[style={-}](A)to(A1);
        \draw[style={-}](A)to(A2);
        \draw[style={-}](A2)to(r1);
        \draw[style={-}](A2)to(r2);
        \draw[style={-}](A1)to(r1);
        \draw[style={-}](B)to(r2);
        \draw[style={-}](B)to(B11);
        \draw[style={-}](A1)to(A11);
        \draw[style={-}](r1)to(r22);
        \draw[style={-}](r2)to(r12);


        \draw (8+0,0)                 node[mi]    (root1)      {};
        \draw (8+0,-1)                 node[mi]    (root)      {};
        \draw (8+-0.75,-2)                 node[mi]    (A)      {};
        \draw (8+-1.5,-3)                 node[mi]    (A1)      {};
        \draw (8+1.5,-3)                 node[mi]    (B)      {};
        \draw (8+0,-3)                 node[mi]    (A2)      {};
        \draw (8+-0.5,-4)                 node[bx]    (r1)      {};
        \draw (8+0.5,-4.75)                 node[mi]    (r12)      {\scriptsize $3$};
        \draw (8+1.1,-4.75)                 node[mi]    (r13)      {\scriptsize $2$};
        \draw (8+-0.5,-4.75)                 node[mi]    (r22)      {\scriptsize $4$};
        \draw (8+-2,-4.75)                 node[mi]    (A11)     {\scriptsize $1$};
        \draw (8+2,-4.75)                 node[mi]    (B11)     {\scriptsize $5$};

        \draw[style={-}](root1)to(root);
        \draw[style={-}](A)to(root);
        \draw[style={-}](B)to(root);
        \draw[style={-}](A)to(A1);
        \draw[style={-}](A)to(A2);
        \draw[style={-}](A2)to(r1);
        \draw[style={-}](A1)to(r1);
        \draw[style={-}](B)to(B11);
        \draw[style={-}](A1)to(A11);
        \draw[style={-}](r1)to(r22);
        \draw[style={-}](A2)to(r12);
        \draw[style={-}](B)to(r13);

 \node (label) at (0.35,-3.1) {\tiny $u$};
 \node (label) at (0,-6) {\small$N_1$};

 \node (label) at (8,-6) {\small$N_2$};
    \end{tikzpicture}

   \caption{\small Two phylogenetic networks $N_1$ and $N_2$ on $[4]$ and $[5]$, respectively. Square nodes represent reticulations, while circular nodes represent internal tree vertices (the root and leaves are also drawn as circles), with arcs directed downwards. Observe that $N_1$ is not tree-child, since the vertex $u$ has only reticulation children, whereas $N_2$ is tree-child.}
    \label{fig:phyloexample}
\end{figure}

A phylogenetic network \(N\) is \emph{tree-child} if every internal vertex has at least one child that is not a reticulation vertex \cite{cardona2008comparison}.  
We denote by \(\TC_{n,k}\) the class of tree-child networks with \(n\) leaves and \(k\) reticulations.


The class of spinal tree-child networks, denoted $\STC_{n,k}$, consists of tree-child networks that contain a \emph{spine}, i.e., a path from the root to a leaf that passes through all internal vertices \cite{francis2024phylogenetic}. Note that since the spine contains all internal vertices, the only nodes outside the spine are the leaves, but notice that the last vertex of the spine is also a leaf. Hence the nodes outside the spine are exactly the set of leaves, except for the leaf that is the last vertex of the spine.
Another fact that will be used (even without further mention) is that each tree vertex has two children, one of which is its next vertex in the spine, and the other one might be a leaf or a reticulation.

\subsection{Previous work}


We recall here the two complementary lines of previous work that motivate our approach. The first is the word-based formulation of the counting problem for general tree-child networks due to Pons and Batle \cite{pons2021combinatorial}. The second is the exact enumeration of spinal tree-child networks obtained in \cite{francis2025counting}. 

\subsubsection{Tree-child networks}

We focus here on the word formulation introduced in \cite{pons2021combinatorial}, which relates the counting of tree-child networks to a class of restricted words. 


Let \(w\) be a word over the alphabet \(\{a_i\}_{i=1}^n\), and denote by \(\#(w,a_j)\) the number of occurrences of \(a_j\) in \(w\).

Let \(\cC_{n,k}\) be the class of words over \(\{a_i\}_{i=1}^n\) satisfying:
\begin{itemize}
    \item \(w\) has length \(2n + k\);
    \item exactly \(k\) symbols occur three times in \(w\), and the remaining \(n-k\) occur twice;
    \item for every prefix \(z\) of \(w\) and every \(i\), either \(\#(z, a_i)=0\) or \(\#(z,a_i) \ge \#(z,a_j)\) for all \(j \ge i\).
\end{itemize}

It is worth noting that, for such a word \(w \in \cC_{n,k}\) to satisfy the third condition, it must hold that the subset of letters \(\{a_i\}_{i=1}^{k}\) corresponds to those that appear three times in \(w\), while the remaining \(\{a_j\}_{j=k+1}^n\) appear twice.

\begin{conj}[Conjecture 1 of \cite{pons2021combinatorial}]\label{conj}
    For $n\geq 1$ and $k\geq 0$, the cardinality of $\TC_{n,k}$ of tree-child networks with $n$ leaves and $k$ reticulation
nodes is related to the cardinality of the class of words $\cC_{n-1,k}$ by the equality

$$|\TC_{n,k}| = \frac{n!}{(n-k)!}\cdot |\cC_{n-1,k}|$$

\end{conj}

This correspondence was proven in \cite{pons2021combinatorial} for trees (\(k=0\)) and maximally reticulated networks (\(k=n-1\)), and it was verified computationally for small parameters. The general case was left open in \cite{pons2021combinatorial}; very recently, a preprint has appeared that claims a proof of the conjecture \cite{lin2026proof}.

\subsubsection{Spinal networks}

We now recall the counting framework for spinal tree-child networks from \cite{francis2025counting}. Let $\NLSTC_{n,k}$ denote the class of spinal tree-child networks with unlabeled leaves (denoted $S(R_{n,k})$ in \cite{francis2025counting}). A first key step is the relation between the labeled and unlabeled classes.


\begin{lem}[Lemma 3.1 of \cite{francis2025counting}]\label{lem:francis2025counting}
For $n\geq 1$ and $k\geq 0$, we have
\[
|\STC_{n,k}| = n!\,|\NLSTC_{n,k}| \;-\; \frac{n!}{2}\,|\NLSTC_{n-1,k}|.
\]
\end{lem}

In \cite{francis2025counting}, the authors construct a bijection between the class of non-labeled spinal tree-child networks $\NLSTC_{n,k}$ and a combinatorial class $\cB_{n-1,k}$. The elements of $\cB_{n,k}$ can be interpreted as partitions of a set of $n+k$ elements into $(n-k)$ singleton blocks and $k$ blocks of size two. The cardinality of this class coincides with the coefficients of the Bessel polynomials and admits the explicit formula
\[
|\cB_{n-1,k}| = \frac{(n-1+k)!}{2^k (n-1-k)! \, k!}.
\]
This correspondence allows one to express the enumeration of non-labeled spinal tree-child networks in purely combinatorial terms. When combined with Lemma~\ref{lem:francis2025counting},  it yields a closed formula for the number of \emph{labeled} spinal tree-child networks, $|\STC_{n,k}|$.

\begin{thm}[Theorem 3.10 of \cite{francis2025counting}]\label{teo:stccount}
Let $\STC_{n,k}$ be the set of binary spinal tree-child networks on $n$ leaves with $k$ reticulations. For $n > 1$ and $k\geq 0$, we have
\[
|\STC_{n,k}| \;=\; n! \, |\cB_{n-1,k}| \;-\; \tfrac{n!}{2}\, |\cB_{n-2,k}|,
\]
where $|\STC_{1,k}| = 1$, and $|\cB_{0,k}| = 0$. Equivalently,

$$|\STC_{n,k}| = \frac{n!(n-2+k)!(n-1+3k)}{2^{k+1}k!(n-1-k)!}.$$
\end{thm}

\section{Spinal networks and word encodings}\label{sec:words}

This section develops word encodings for spinal tree-child networks. We first introduce a restricted word class that encodes unlabeled spinal networks and yields explicit counting formulas. We then relate this encoding to the alternative encoding from \cite{francis2025counting}, and finally extend the word approach to spinal caterpillar tree-child networks.

\subsection{Word encoding and enumeration of spinal networks}\label{subsec:word-encoding}


We start by encoding networks in $\NLSTC_{n,k}$ as words lying in $C_{n,k}$. We define the subclass $C_{n,k}^1\subset C_{n,k}$ as the class of words $w\in C_{n,k}$ such that the following holds:

\begin{itemize}
    \item[1)] $w$ has suffix $a_{k+1}a_{k+2}\ldots a_n$, and
    \item[2)] For every prefix $z$ of $w$, and for every $i\leq k$, if $\#(z,a_i) = 2$, then $za_i$ is also a prefix of $w$.
\end{itemize}

Informally, the subclass $C_{n,k}^1$ contains words in which, for each letter that appears thrice, the last two appearances must be consecutive, and for the letters that appear twice, the second must appear at the end of the word.

Let $w_1,w_2\in C^1_{n,k}$. We define the equivalence relation $w_1\sim w_2$ if and only if there exists a permutation
\[
\sigma\in S_{\{k+1,\dots,n\}}\cong S_{\,n-k}
\]
such that $w_2$ is obtained from $w_1$ by replacing the first appearances of $a_i$ with $a_{\sigma(i)}$ for all $i\in\{k+1,\dots,n\}$ (and leaving the rest of the word unchanged). Equivalently, $w_1\sim w_2$ when $w_2$ is obtained from $w_1$ by permuting the first appearances of the letters $a_{k+1},\dots,a_n$.

Before establishing the bijection between $\NLSTC_{n,k}$ and the word class $C_{n-1,k}^1 / {\sim}$, it is useful to understand the structure of the spine in a non-labeled spinal tree-child network. 
In particular, the number of leaves and reticulations determines both the number of tree vertices and the length of the spine, as the following lemma shows.

\begin{lem}\label{lem:lengthspine}
    
    Let $N \in \NLSTC_{n,k}$, and let $t,\ell$ denote respectively the number of tree nodes and the length of the spine of $N$. Then, $t=n+k-1$ and $\ell = n+2k$.
\end{lem}

\begin{proof}
The number of arcs of $N$ can be counted either by counting outgoing arcs, which gives $1+2t+k$ (corresponding to the root, tree vertices and reticulations), or by counting incoming arcs, which gives $t+2k+n$ (corresponding to tree vertices, reticulations and leaves). Since the two quantities must be equal, we get the identity $t=k+n-1$. Now, the length of the spine equals $t+k+1=n+2k$.
\end{proof}

The following result establishes a correspondence between the sets $\NLSTC_{n,k}$ and $C_{n-1,k}^1 / {\sim}$. The proof follows the approach introduced in \cite{pons2021combinatorial}, encoding each network as a word via a decomposition into suitable components. Specifically, we first decompose the network into \emph{path components}, namely, paths $(v_0, \ldots, v_m)$ such that each internal vertex $v_i$ ($1 \le i < m$) is a tree node and $v_m$ is a leaf. We then order these paths according to their position along the spine and assign labels to the vertices of $N$ from a fixed alphabet. Finally, by reading the labeled paths in this prescribed order, we obtain a word in $C_{n,k}$ that satisfies the additional constraints defining $C_{n,k}^1$.

\begin{prop}\label{prop:encod}
    There is a bijection between $\NLSTC_{n,k}$ and $C_{n-1,k}^1/{\sim}$.
\end{prop}

\begin{proof}
    We first show that every network in $\NLSTC_{n,k}$ can be encoded as a word in $C_{n-1,k}^1$, and then we verify that the resulting encoding is well-defined on the quotient $C_{n-1,k}^1 / {\sim}$. Moreover, we show that any such word corresponds to a network in $\NLSTC_{n,k}$, establishing the desired bijection.

    The encoding procedure is depicted in Figure~\ref{fig:uldecomp} and is composed of different steps, which we now describe in detail.

    \begin{itemize}

    \item \textbf{Step 1.} Let $N \in \NLSTC_{n,k}$, and let $s = (u_0, u_1, \ldots, u_\ell)$ denote the spine of $N$. Since $s$ contains all internal vertices of the network, it in particular includes all $k$ reticulation nodes. Denote the reticulations by $u_{\ell_i}$, for $i=1,\dots,k$, ordered by their appearance along the spine. For each reticulation $u_{\ell_i}$, consider the (tree) node that precedes it in the spine, that is, $u_{\ell_i - 1}$; this node must have a child, which cannot be on the spine (in that case, this child would have two different parents, contradicting the tree-child condition). Hence, this child must be a leaf, which we denote by $v_{\ell_i}$. Notice that not all the leaves of $N$ are of this form. Indeed, there are $n-1$ leaves (excluding the leaf at the end of the spine), and we have identified $k$ of them as the children of the specified parents of reticulations, and hence there are $n-k-1$ remaining leaves. We denote these remaining leaves by $w_1,\dots,w_{n-k-1}$ (where the ordering of these nodes is arbitrary) and call them exterior leaves. See panel (b) of Figure~\ref{fig:uldecomp} for an illustration of this labeling of the vertices of $N$.

    \item \textbf{Step 2.}
    Based on the definitions given in the previous step, we define the \emph{path components} $P_0,\dots,P_{n-1}$ of $N$ as follows:
    \[
    \begin{aligned}
    P_0 &:= (u_0,u_1,\dots,u_{\ell_1-1},v_{\ell_1}),\\
    P_i &:= (u_{\ell_i},u_{\ell_i+1},\dots,u_{\ell_{i+1}-1},v_{\ell_{i+1}}) &&\text{for }1\le i\le k-1,\\
    P_k &:= (u_{\ell_k},u_{\ell_k+1},\dots,u_l)\\
    P_{k+j} & := (w_j) &&\text{for }1\le j\le n-k-1,
    \end{aligned}
    \]
    Note that all these paths are vertex-disjoint, and they cover all vertices of $N$. Also, all the arcs in the paths belong to the original network, but there are arcs in $N$ that do not belong to any of the paths. Namely, the only arcs that do not belong to any path component are the reticulation arcs and the arcs that lead to the leaves $w_1,\dots,w_{n-k-1}$.
    See panel (c) of Figure~\ref{fig:uldecomp} for an illustration of the path components of a network.

    \item \textbf{Step 3.} We now assign labels from the alphabet $\{a_1, \dots, a_{n-1}\}$ to some of the vertices of $N$. We emphasize that the labeling is not injective (different nodes will get the same label) and not all vertices of $N$ will be labeled. The labeling is defined as follows:
    \begin{itemize}
        \item For each $i=1,\dots,k$, the reticulation vertex $u_{\ell_i}$, and both of its parents are assigned the label $a_i$.
        \item For each $j=1, \ldots, n-k-1$, the leaf $w_j$  and its single parent (which is a tree node), are assigned the label $a_{k+j}$.
        \item All other nodes are left unlabeled.
    \end{itemize}
    It is straightforward to verify that this labeling is well-defined, since each tree node is either the parent of a single reticulation, or the parent of a leaf, but not both simultaneously, and hence it will receive exactly one label. Also, each reticulation receives exactly one label, and each leaf receives at most one label.
    Note also that the only nodes that do not receive a label are the leaves belonging to the non-trivial path components $P_0,\dots,P_{k-1}$, and the root of the network.
    This labeling is illustrated in panel (d) of Figure~\ref{fig:uldecomp}.
\end{itemize}

    Now, we associate to each path component $P_i$ the word $\tilde w_i$ obtained by reading the labels of the (labeled) vertices in $P_i$, in the order that they appear in the path. Finally, we concatenate these words in order to obtain a word $w=\tilde w_0\tilde w_1\cdots \tilde w_{n-1}$. 
    
    We verify that $w$ belongs to $C_{n-1,k}^1$. From the definition of the labeling it is clear that the letters $a_1,\dots,a_k$ appear three times in $w$, while the remaining letters $a_{k+1},\dots,a_{n-1}$ appear twice. 
    Also, it is clear that $a_{k+1}a_{k+2}\cdots a_{n-1}$ is a suffix of $w$, since the last $n-k-1$ path components consist of the leaves $w_1,\dots,w_{n-k-1}$, which are labeled with these letters. 
    Finally, for each $i\le k$, if $\#(z,a_i)=2$ for some prefix $z$ of $w$, then the last appearance of $a_i$ corresponds to the upper parent of the reticulation $u_{\ell_i}$ (with respect to the spine order), and these two vertices are adjacent in the spine, so they belong to consecutive path components, and since the leaf child of $u_{\ell_i-1}$ is not labeled, the prefix $z$ must be followed by the labeled assigned to $u_{\ell_i}$. Thus, the last two appearances of $a_i$ are consecutive in $w$. Therefore, $w\in C_{n-1,k}^1$.

    Since the labeling of the leaves $w_1,\dots,w_{n-k-1}$ is arbitrary, the resulting word $w$ is well-defined only up to the equivalence relation ${\sim}$. Hence, we have shown that every network in $\NLSTC_{n,k}$ can be encoded as a word in $C_{n-1,k}^1 / {\sim}$.

    Conversely, given a word $w\in C_{n-1,k}^1$, we can reverse the previous construction to obtain a network in $\NLSTC_{n,k}$ that encodes $w$. 
    
    \begin{itemize}
        \item \textbf{Step 1.} The last $n-k-1$ letters of $w$ are $a_{k+1}a_{k+2}\cdots a_{n-1}$, and they correspond to the subwords $\tilde w_{k+1},\dots,\tilde w_{n-1}$, which consist of a single letter each. As for the first $k$ subwords $\tilde w_0,\dots,\tilde w_{k-1}$, they can be determined by splitting $w$ between the second and third appearances of the letters $a_1,\dots,a_k$, which are consecutive in $w$ by the definition of $C_{n-1,k}^1$.
        \item \textbf{Step 2.} We can trivially reconstruct the path components $P_0,\dots,P_{n-1}$ from these subwords, labeling their nodes according to the letters of the subwords $\tilde w_0,\dots,\tilde w_{n-1}$. 
        \item \textbf{Step 3.} Finally, we can reconstruct the network by connecting each internal vertex in a path, say it is labeled by $a_i$, with the single first node in a path that shares the same label $a_i$.
    \end{itemize}

    It is straightforward to verify that the resulting network is a spinal tree-child network, and that it is encoded by the original word $w$. Hence, we have shown that every word in $C_{n-1,k}^1$ corresponds to a network in $\NLSTC_{n,k}$.
\end{proof}

We illustrate in Figure~\ref{fig:uldecomp} the encoding of a non-labeled spinal tree-child network, with \(n = 5\) leaves and \(k = 2\) reticulations.
Panel~(a) shows the unlabeled network. In panel~(b), the internal vertices along the spine are denoted according to their roles in the proof, and in panel~(c) the network is decomposed into the corresponding path components \(P_0, \ldots, P_4\).  
In panel~(d), we assign the alphabet labels as prescribed in the proof, so we can read the vertices following the path component order, obtaining the 
subwords \(a_3a_1a_2a_1\), \(a_1a_2\), \(a_2a_4\), \(a_3\), \(a_4\) corresponding to the path components
and finally the
word  
\[
w = a_3a_1a_2a_1a_1a_2a_2a_4a_3a_4 \in C_{4,2}^1.
\]
Exchanging the singleton path components \(P_3\) and \(P_4\) leads to an alternative encoding, yielding  
\[
w' = a_4a_1a_2a_1a_1a_2a_2a_3a_3a_4 \in C_{4,2}^1,
\]
which satisfies \(w \sim w'\).  
This illustrates how networks differing only by the labeling of the exterior leaves correspond to equivalent words under the equivalence relation \(\sim\).
Conversely, if we are given the word \(w\) as above, by counting the number of different letters and their multiplicities, we find that it must belong to \(C_{4,2}^1\); hence the last two subwords must be the single letters \(a_3\) and \(a_4\). The remaining subwords are obtained by splitting the remaining letters between the second and third occurrences of $a_1$ and $a_2$. Hence, the full sequence of subwords is \(a_3a_1a_2a_1\), \(a_1a_2\), \(a_2a_4\), \(a_3\), \(a_4\) and from these we can reconstruct the path components. Finally, by adding the arcs from the first and second occurrences of $a_1$ and $a_2$, to the third one, and from the first one of $a_3$ and $a_4$ to the second one, we reconstruct the original network.

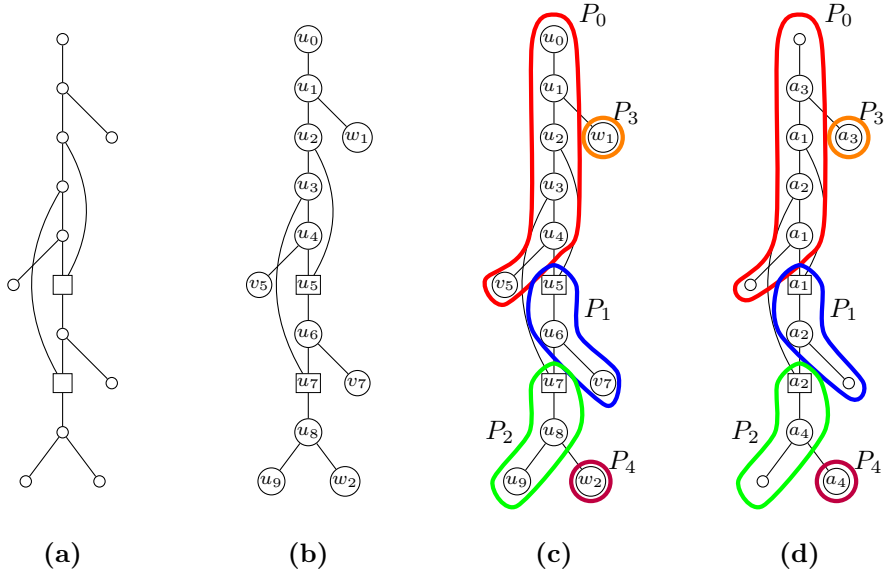
\begin{figure}[h]
\centering

\begin{tikzpicture}[scale=0.65]

\node[mi] (a) at    (-5+0,0)        {};
\node[mi] (b) at    (-5+0,-1)       {};
\node[mi] (b1) at   (-5+1,-2)       {};
\node[mi] (c) at    (-5+0,-2)       {};
\node[mi] (d) at    (-5+0,-3)       {};
\node[mi] (e) at    (-5+0,-4)       {};
\node[mi] (e1) at   (-5+-1,-5)      {};
\node[bxl] (f) at   (-5+0,-5)       {};
\node[mi] (g) at    (-5+0,-6)       {};
\node[mi] (g1) at   (-5+1,-7)       {};
\node[bxl] (h) at   (-5+0,-7)       {};
\node[mi] (i) at    (-5+0,-8)       {};
\node[mi] (j) at    (-5+0.75,-9)    {};
\node[mi] (j1) at   (-5+-0.75,-9)   {};

\node[label] at (-5,-10.5) {\textbf{(a)}};
\node[label] at (0,-10.5) {\textbf{(b)}};
\node[label] at (5,-10.5) {\textbf{(c)}};
\node[label] at (10,-10.5) {\textbf{(d)}};

\draw (a) -- (b) -- (c) -- (d) -- (e) -- (f) -- (g) -- (h) -- (i) -- (j);
\draw (i) -- (j1);
\draw (b) -- (b1);
\draw (e) -- (e1);
\draw (g) -- (g1);
\draw (c) to[bend left] (f);
\draw (d) to[bend right] (h);
\node[mi] (a) at (0,0) {\scriptsize $u_0$};
\node[mi] (b) at (0,-1) {\scriptsize $u_1$};
\node[mi] (b1) at (1,-2) {\scriptsize $w_1$};
\node[mi] (c) at (0,-2) {\scriptsize $u_2$};
\node[mi] (d) at (0,-3) {\scriptsize $u_3$};
\node[mi] (e) at (0,-4) {\scriptsize $u_4$};
\node[mi] (e1) at (-1,-5) {\scriptsize $v_5$};
\node[bxl] (f) at (0,-5) {\scriptsize $u_5$};
\node[mi] (g) at (0,-6) {\scriptsize $u_6$};
\node[mi] (g1) at (1,-7) {\scriptsize $v_7$};
\node[bxl] (h) at (0,-7) {\scriptsize $u_7$};
\node[mi] (i) at (0,-8) {\scriptsize $u_8$};
\node[mi] (j) at (0.75,-9) {\scriptsize $w_2$};
\node[mi] (j1) at (-0.75,-9) {\scriptsize $u_9$};

\node[label] at (0,-10.5) {\textbf{(b)}};
\node[label] at (5,-10.5) {\textbf{(c)}};
\node[label] at (10,-10.5) {\textbf{(d)}};

\node[label] at (5.8,0.5) {$P_0$};
\node[label] at (5.9,-5.5) {$P_1$};
\node[label] at (3.9,-8) {$P_2$};
\node[label] at (6.45,-1.5) {$P_3$};
\node[label] at (6.4,-8.6) {$P_4$};

\node[label] at (5+5.8,0.5) {$P_0$};
\node[label] at (5+5.9,-5.5) {$P_1$};
\node[label] at (5+3.9,-8) {$P_2$};
\node[label] at (5+6.45,-1.5) {$P_3$};
\node[label] at (5+6.4,-8.6) {$P_4$};

\draw (a) -- (b) -- (c) -- (d) -- (e) -- (f) -- (g) -- (h) -- (i) -- (j);
\draw (i) -- (j1);
\draw (b) -- (b1);
\draw (e) -- (e1);
\draw (g) -- (g1);
\draw (c) to[bend left] (f);
\draw (d) to[bend right] (h);

\node[mi] (a) at    (5+0,0) {\scriptsize $u_0$};
\node[mi] (b) at    (5+0,-1) {\scriptsize $u_1$};
\node[mi] (b1) at   (5+1,-2) {\scriptsize $w_1$};
\node[mi] (c) at    (5+0,-2) {\scriptsize $u_2$};
\node[mi] (d) at    (5+0,-3) {\scriptsize $u_3$};
\node[mi] (e) at    (5+0,-4) {\scriptsize $u_4$};
\node[mi] (e1) at   (5+-1,-5) {\scriptsize $v_5$};
\node[bxl] (f) at   (5+0,-5) {\scriptsize $u_5$};
\node[mi] (g) at    (5+0,-6) {\scriptsize $u_6$};
\node[mi] (g1) at   (5+1,-7) {\scriptsize $v_7$};
\node[bxl] (h) at   (5+0,-7) {\scriptsize $u_7$};
\node[mi] (i) at    (5+0,-8) {\scriptsize $u_8$};
\node[mi] (j) at    (5+0.75,-9) {\scriptsize $w_2$};
\node[mi] (j1) at   (5+-0.75,-9) {\scriptsize $u_9$};

\draw (a) -- (b) -- (c) -- (d) -- (e) -- (f) -- (g) -- (h) -- (i) -- (j);
\draw (i) -- (j1);
\draw (b) -- (b1);
\draw (e) -- (e1);
\draw (g) -- (g1);
\draw (c) to[bend left] (f);
\draw (d) to[bend right] (h);


\draw [red, ultra thick] plot [smooth cycle] coordinates {(5+0,0.5) (5+0.45,-0.1) (5+0.5,-3.7) (5+0.1,-4.4) (5+-1,-5.4) (5+-1.4,-4.9)(5+-0.5,-4) (5+-0.45,-0.1)};
\draw [blue, ultra thick] plot [smooth cycle] coordinates {(5+0,-4.6)(5+0.45,-5)(5.5,-6)(6.3,-6.8)(6.05,-7.45)(4.6,-6)(4.55,-5)};
\draw [green, ultra thick] plot [smooth cycle] coordinates {(5,-6.6)(5.45,-7)(5.5,-8)(4.25,-9.5)(3.75,-9)(4.5,-8)(4.55,-7)};
\draw[orange, ultra thick] (b1) circle (0.4cm);
\draw[purple, ultra thick] (j) circle (0.4cm);

\node[mi] (a) at    (10+0,0) {\scriptsize };
\node[mi] (b) at    (10+0,-1) {\scriptsize $a_3$};
\node[mi] (b1) at   (10+1,-2) {\scriptsize $a_3$};
\node[mi] (c) at    (10+0,-2) {\scriptsize $a_1$};
\node[mi] (d) at    (10+0,-3) {\scriptsize $a_2$};
\node[mi] (e) at    (10+0,-4) {\scriptsize $a_1$};
\node[mi] (e1) at   (10+-1,-5) {};
\node[bxl] (f) at   (10+0,-5) {\scriptsize $a_1$};
\node[mi] (g) at    (10+0,-6) {\scriptsize $a_2$};
\node[mi] (g1) at   (10+1,-7) {};
\node[bxl] (h) at   (10+0,-7) {\scriptsize $a_2$};
\node[mi] (i) at    (10+0,-8) {\scriptsize $a_4$};
\node[mi] (j) at    (10+0.75,-9) {\scriptsize $a_4$};
\node[mi] (j1) at   (10+-0.75,-9) {\scriptsize };

\draw (a) -- (b) -- (c) -- (d) -- (e) -- (f) -- (g) -- (h) -- (i) -- (j);
\draw (i) -- (j1);
\draw (b) -- (b1);
\draw (e) -- (e1);
\draw (g) -- (g1);
\draw (c) to[bend left] (f);
\draw (d) to[bend right] (h);


\draw [red, ultra thick] plot [smooth cycle] coordinates {(5+5+0,0.5) (5+5+0.45,-0.1) (5+5+0.5,-3.7) (5+5+0.1,-4.5) (5+5+-1,-5.3) (5+5+-1.3,-5)(5+5+-0.5,-4) (5+5+-0.45,-0.1)};
\draw [blue, ultra thick] plot [smooth cycle] coordinates {(5+5+0,-4.6)(5+5+0.45,-5)(5+5.5,-6)(5+6.3,-7)(5+6,-7.3)(5+4.6,-6)(5+4.55,-5)};
\draw [green, ultra thick] plot [smooth cycle] coordinates {(5+5,-6.6)(5+5.45,-7)(5+5.5,-8)(5+4.25,-9.5)(5+3.75,-9)(5+4.5,-8)(5+4.55,-7)};
\draw[orange, ultra thick] (b1) circle (0.4cm);
\draw[purple, ultra thick] (j) circle (0.4cm);

\end{tikzpicture}
\caption{\small 
Example of the encoding process described in the proof of Proposition \ref{prop:encod} for a spinal tree-child network with \(n=5\) and \(k=2\): 
\textbf{(a)} unlabeled network, 
\textbf{(b)} identification of nodes, 
\textbf{(c)} decomposition into path components, and 
\textbf{(d)} partial labeling of the nodes over the alphabet $\{a_1,a_2,a_3,a_4\}$.
}
\label{fig:uldecomp}

\end{figure}

\begin{prop} \label{prop:wordsbessel}
For \(n\geq 1\) and \(k\leq n\) we have
\[
|C_{n,k}^1/{\sim}| \;=\; \frac{(n+k)!}{2^k (n-k)!\, k!}.
\]
\end{prop}

\begin{proof}
Let \(w \in C_{n,k}^1\). Recall that the letters \(a_1, \ldots, a_k\) appear three times in \(w\), with the last two occurrences of each letter consecutive, while the letters \(a_{k+1}, \ldots, a_n\) appear twice and their second occurrences form a fixed suffix. 

Removing this fixed suffix leaves a word of length \(n + 2k\), containing three occurrences of each \(a_i\) for \(1 \le i \le k\) and one occurrence of each \(a_j\) for \(k+1 \le j \le n\).  
Now, for each \(a_i\) with \(i \le k\), the last two consecutive occurrences can be viewed as a single block. Compressing these consecutive pairs into single blocks reduces the length by \(k\), yielding a new sequence \(w^\ast\) of length \(n + k\).  
Each compressed block represents a paired structure corresponding to the \(a_i\) with \(i \le k\), and each of the remaining \(n-k\) positions corresponds to a singleton letter among \(a_{k+1}, \ldots, a_n\).  

Hence, \(w^\ast\) determines a partition of the set of positions \(\{1, 2, \ldots, n+k\}\) into \(k\) unordered pairs, corresponding to the first occurrences and their associated blocks of each \(a_i\), and \(n-k\) singletons, corresponding to the remaining letters.  
Conversely, given any such partition, we can reconstruct a unique class in \(C_{n,k}^1/{\sim}\). Each pair is assigned to one of the letters \(a_1, \ldots, a_k\) according to the order of their maximal elements: the pair whose maximal element is smallest is associated with \(a_1\), the next with \(a_2\), and so on.  
Once this correspondence is fixed, the singletons are filled with the remaining letters \(a_{k+1}, \ldots, a_n\), and the fixed suffix \(a_{k+1}\ldots a_n\) is reattached.  
This establishes a bijection between \(C_{n,k}^1/{\sim}\) and the set \(\cB_{n,k}\) of partitions of an \((n+k)\)-element set into \(k\) unordered pairs and \(n-k\) singletons.

The number of such partitions is well known. One first chooses the \(2k\) elements that will form the pairs, in \(\binom{n+k}{2k}\) ways, and then divides these \(2k\) elements into \(k\) unordered pairs in \({(2k)!}/({2^k k!})\) ways. The remaining \(n-k\) elements are singletons. Therefore,
\[
|\cB_{n,k}| = \binom{n+k}{2k}\frac{(2k)!}{2^k k!}
= \frac{(n+k)!}{2^k (n-k)! \, k!}.
\]
Since \(|C_{n,k}^1/{\sim}| = |\cB_{n,k}|\), the result follows.
\end{proof}

By combining Propositions \ref{prop:encod} and \ref{prop:wordsbessel}, we obtain an explicit enumeration formula for the class \( \NLSTC_{n,k} \).  
In turn, when this result is combined with Lemma \ref{lem:francis2025counting}, it yields Theorem \ref{teo:stccount}.

\begin{coro}
    For $n\geq 1$ and $k\leq n$ we have

    $$|\NLSTC_{n,k}| = |C_{n-1,k}^1/{\sim}| = \frac{(n-1+k)!}{2^k(n-1-k)!k!}.$$
\end{coro}

\subsection{Alternative word encoding for unlabeled spinal tree-child networks}

Let \( N \) be an unlabeled spinal tree-child network, and let \( s = (u_0, u_1, \ldots, u_\ell) \) denote its spine.  
It is worth noting that the encoding of \( N \) into an element of \( C_{n,k}^1 / {\sim} \) follows rules closely related to the characterization of internal vertices described in Section 3.2 of \cite{francis2025counting}.  
In that work, each internal vertex \( v \in s \setminus \{u_0, u_\ell\} \) of an unlabeled spinal tree-child network is labeled  according to the following rules:
\begin{enumerate}
    \item If \( v \) has a leaf child that does not lie on the spine, then \( v \) is labeled by \( L \).
    \item If \( v \) is the \( i \)-th reticulation vertex along the spine, counted from the terminal leaf \( u_\ell \), then \( v \) is labeled by \( R_i \).
    \item If \( v \) is the parent of the \( i \)-th reticulation vertex that is not its immediate ancestor on the spine, then \( v \) is labeled by \( Q_i \).
\end{enumerate}

By reading the labels of the internal vertices along the spine from bottom to top, each network can be encoded by a word over the alphabet \[\{L, R_1, \ldots, R_k, Q_1, \ldots, Q_k\}.\]  
We refer to this representation as the \(\{L, R, Q\}\)-encoding.  
A word in this encoding contains \(n - 1\) occurrences of the letter \(L\) and exactly one occurrence of each of the remaining symbols.

An example of the \(\{L, R, Q\}\)-encoding is shown in Figure~\ref{fig:uldecomp2}, which depicts the same network as in Figure~\ref{fig:uldecomp}, but with its internal vertices labeled according to the rules of this scheme.  
Once the vertices of \(N\) are labeled, the corresponding word is obtained by reading the labels along the spine from bottom to top, which gives \(w = LR_1LR_2LQ_1Q_2L\).

In \cite{francis2025counting}, the same \(\{L, R, Q\}\)-encoding is interpreted in terms of certain set covers, where each internal node corresponds to a subset that captures its local structure along the spine.

To better understand the relationship between both classes of word encodings, we define a transformation that converts one type of word into the other.  
Let \(w\) be the word obtained by applying the \(\{L, R, Q\}\)-encoding over an unlabeled spinal tree-child network. We define a transformation \(T\) that converts $w$ into a word lying in \(C_{n-1,k}^1/{\sim}\) as follows:

\begin{itemize}
    \item[1)] Replace every \(R_i\), every \(Q_i\), and every \(L\) that is immediately consecutive to some \(R_i\), by \(a_{k+1-i}\).
    \item[2)] For the remaining \(n-k-1\) letters \(L\), replace them, from right to left, by \(a_{k+1}, a_{k+2}, \ldots, a_{n-1}\), respectively.
    \item[3)] Reverse  the resulting word.
    \item[4)] Append the suffix \(a_{k+1}a_{k+2}\ldots a_{n-1}\) to the end of the reversed word.
\end{itemize}

The word obtained after Step~(4) is denoted by \(T(w)\). One can check that \(T(w)\in C_{n-1,k}^1/{\sim}\).

To illustrate the transformation, consider the network shown in Figure~\ref{fig:uldecomp2}, whose \(\{L, R, Q\}\)-encoding is 
\[
w = LR_1LR_2LQ_1Q_2L.
\]
Applying the four transformation steps to \(w\), we obtain:
\[
\begin{aligned}
w_1 &= La_2a_2a_1a_1a_2a_1L,\\
w_2 &= a_4a_2a_2a_1a_1a_2a_1a_3,\\
w_3 &= a_3a_1a_2a_1a_1a_2a_2a_4,\\
T(w) &= a_3a_1a_2a_1a_1a_2a_2a_4a_3a_4.
\end{aligned}
\]
Hence, the word associated with the network in Figure~\ref{fig:uldecomp2} is \[T(w) = a_3a_1a_2a_1a_1a_2a_2a_4a_3a_4.\]

\begin{figure}[h]
\centering

\begin{tikzpicture}[scale=0.65]
 
\node[mi] (a) at    (5+0,0) {};
\node[mi] (b) at    (5+0,-1) {\scriptsize $L$};
\node[mi] (b1) at   (5+1,-2) {};
\node[mi] (c) at    (5+0,-2) {\scriptsize $Q_2$};
\node[mi] (d) at    (5+0,-3) {\scriptsize $Q_1$};
\node[mi] (e) at    (5+0,-4) {\scriptsize $L$};
\node[mi] (e1) at   (5+-1,-5) {};
\node[bxl] (f) at   (5+0,-5) {\scriptsize $R_2$};
\node[mi] (g) at    (5+0,-6) {\scriptsize $L$};
\node[mi] (g1) at   (5+1,-7) {};
\node[bxl] (h) at   (5+0,-7) {\scriptsize $R_1$};
\node[mi] (i) at    (5+0,-8) {\scriptsize $L$};
\node[mi] (j) at    (5+0.75,-9) {};
\node[mi] (j1) at   (5+-0.75,-9) {};

\draw (a) -- (b) -- (c) -- (d) -- (e) -- (f) -- (g) -- (h) -- (i) -- (j);
\draw (i) -- (j1);
\draw (b) -- (b1);
\draw (e) -- (e1);
\draw (g) -- (g1);
\draw (c) to[bend left] (f);
\draw (d) to[bend right] (h);

\end{tikzpicture}
\caption{Example of the labeling of interior nodes proposed in \cite{francis2025counting} for the unlabeled spinal tree-child network depicted in Figure \ref{fig:uldecomp} (a).}\label{fig:uldecomp2}
\end{figure}
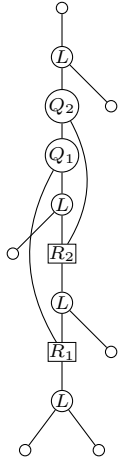

Although the \(\{L,R,Q\}\)-encoding in \cite{francis2025counting} is closely related to our construction, the two encodings emphasise different aspects. The encoding via \(C^{1}_{n,k}/{\sim}\) is designed to connect spinal networks with the word framework appearing in Conjecture~\ref{conj} (and, more broadly, with the class of words considered in \cite{pons2021combinatorial}). 
On the other hand, the \(\{L,R,Q\}\)-encoding in \cite{francis2025counting} fits naturally into the expanding-cover approach developed there.

\subsection{Encoding and counting of spinal caterpillar tree-child networks}

We now consider a natural extension of spinal tree-child networks obtained by allowing the leaves that have a reticulation as a sibling to be replaced by caterpillar structures.  
We refer to these networks as \emph{spinal caterpillar tree-child networks}.  
We denote by $\SCTC_{n,k}$ and $\NLSCTC_{n,k}$ the labeled and non-labeled versions of this class, respectively.  
As before, $n$ denotes the number of leaves and $k$ the number of reticulations.  
Observe that spinal caterpillar tree-child networks form a superclass of spinal tree-child networks.

The goal of this subsection is to obtain counting formulas for these networks by extending the word encoding developed in subsection \ref{subsec:word-encoding}. 
More precisely, we introduce a new subclass of words that relaxes the restrictions defining $C_{n,k}^1$, allowing the separation between the second and third occurrences of certain letters.  
As we will see, the same encoding procedure used for spinal networks leads to a bijection with this larger word class.

\medskip

We define the subclass $C_{n,k}^2 \subset C_{n,k}$ as the class of words $w \in C_{n,k}$ satisfying the following conditions.

\begin{itemize}
    \item[1)] $w$ has suffix $a_{k+1}a_{k+2}\ldots a_n$;

    \item[2)]For every $i\le k$, if $z$ is a prefix of $w$ such that $\#(z,a_i)=2$, then the only letters in $w$ appearing between $z$ and the third occurrence of $a_i$ are letters from $\{a_{k+1},\ldots,a_n\}$.
\end{itemize}

That is, the class $C_{n,k}^2$ consists of words in which the second and third occurrences of a letter $a_i$ with $i\le k$ may be separated, but only by letters that appear exactly twice in the word.  
In contrast with the class $C_{n,k}^1$, where the last two occurrences of each such letter must be consecutive, here these occurrences may be separated by a block formed exclusively by letters corresponding to exterior leaves.

We consider the equivalence relation on $C_{n,k}^2$ defined by permuting the first occurrences of the letters $a_{k+1},\ldots,a_n$.  
More precisely, for $w_1,w_2\in C_{n,k}^2$ we write $w_1\sim w_2$ if there exists a permutation
\[
\sigma \in S_{\{k+1,\ldots,n\}}
\]
such that $w_2$ is obtained from $w_1$ by replacing the first occurrences of $a_i$ by $a_{\sigma(i)}$ for all $i\in\{k+1,\ldots,n\}$.

The following result is the analogue of Proposition~\ref{prop:encod} for this extended class of networks.

\begin{prop}\label{prop:caterpillar_bijection}
There is a bijection between $\NLSCTC_{n,k}$ and $C_{n-1,k}^2/{\sim}$.
\end{prop}

\begin{proof}

    The proof follows the same lines as the proof of Proposition~\ref{prop:encod}, with the necessary adjustments to account for caterpillars outside the spine and the relaxed conditions defining $C_{n-1,k}^2$. We shall not go into the same level of detail as in the proof of Proposition~\ref{prop:encod}, but we will highlight the key differences and how they are addressed.

    First, the decomposition of the network into path components made in the proof of Proposition~\ref{prop:encod} can be read as follows: we start from the root (or a reticulation vertex) and follow the spine downwards until we would reach a reticulation vertex; at this point, we stop following the spine and instead visit the leaf that is the sibling of the reticulation vertex. In case of a caterpillar spinal network, we follow the same procedure, but instead of finding a leaf that is the sibling of a reticulation vertex, we may find a caterpillar structure. In this case, we follow a maximal path along the caterpillar until we reach a leaf. As in the previous case, we may find isolated leaves that do not belong to any of the path components found so far, and we add them as path components as well. See Figure~\ref{fig:spinalcater} for an example of this decomposition in the case of a spinal caterpillar network.

    Second, the labeling of the vertices is done in the same way as in the proof of Proposition~\ref{prop:encod}. Namely, we first label the reticulations (with the letters $a_1,\ldots,a_k$) and those leaves that form a (trivial) path component (with the letters $a_{k+1},\ldots,a_{n-1}$). Then, we label the (one or two) parents of each node labeled so far with the same letter as the node itself. See Figure~\ref{fig:spinalcater} for an example of this labeling.

    Finally, the word encoding the network is obtained in exactly the same way as in the proof of Proposition~\ref{prop:encod}, by reading the labels along the path components, in the order they appear along the spine and later the isolated leaves. The resulting word satisfies the relaxed conditions defining $C_{n-1,k}^2$ because the only letters that can appear between the second and third occurrences of a letter $a_i$ with $i\le k$ are those corresponding to leaves that belong to the caterpillars, and hence they are of the form $a_j$ for some $j>k$. As before, the only choice that is arbitrarily made is the order in which the trivial path components are read, which corresponds to the equivalence relation $\sim$ defined on $C_{n,k}^2$.

Conversely, given a word $w\in C_{n-1,k}^2/{\sim}$, the corresponding network can be reconstructed analogously to the process given in the proof of Proposition~\ref{prop:encod}. Indeed, $w$ will have a suffix of the form $a_{k+1}a_{k+2}\ldots a_{n-1}$, which determines the labels of the leaves that form trivial path components. The corresponding prefix of $w$ can be split into words that determine the path components, simply by cutting the word just before the third occurrence of each letter $a_i$ with $i\le k$. The reconstruction of the path components and the network itself is done in exactly the same way as in the proof of Proposition~\ref{prop:encod}. 
\end{proof}

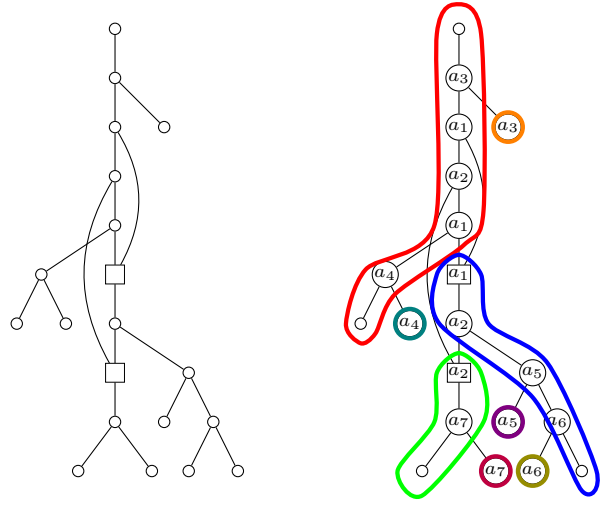
\begin{figure}[h]
\centering

\begin{tikzpicture}[scale=0.65]

\node[mi] (a) at    (-2+0,0)        {};
\node[mi] (b) at    (-2+0,-1)       {};
\node[mi] (b1) at   (-2+1,-2)       {};
\node[mi] (c) at    (-2+0,-2)       {};
\node[mi] (d) at    (-2+0,-3)       {};
\node[mi] (e) at    (-2+0,-4)       {};
\node[mi] (e1) at   (-2+-1.5,-5)      {};
\node[mi] (e11) at  (-2+-2,-6)      {};
\node[mi] (e12) at  (-2+-1,-6)      {};
\node[bxl] (f) at   (-2+0,-5)       {};
\node[mi] (g) at    (-2+0,-6)       {};
\node[mi] (g1) at   (-2+1.5,-7)       {};
\node[mi] (g11) at  (-2+1.5+.5,-8)       {};
\node[mi] (g12) at  (-2+1-.5+.5,-8)       {};
\node[mi] (g111) at (-2+1.5+.5*2,-9)       {};
\node[mi] (g112) at (-2+1.5,-9)       {};
\node[bxl] (h) at   (-2+0,-7)       {};
\node[mi] (i) at    (-2+0,-8)       {};
\node[mi] (j) at    (-2+0.75,-9)    {};
\node[mi] (j1) at   (-2+-0.75,-9)   {};


\draw (a) -- (b) -- (c) -- (d) -- (e) -- (f) -- (g) -- (h) -- (i) -- (j);
\draw (i) -- (j1);
\draw (b) -- (b1);
\draw (e) -- (e1);
\draw (e11) -- (e1);
\draw (e12) -- (e1);
\draw (g) -- (g1);
\draw (g11) -- (g1);
\draw (g12) -- (g1);
\draw (g11) -- (g111);
\draw (g11) -- (g112);
\draw (c) to[bend left] (f);
\draw (d) to[bend right] (h);

\node[mi] (a) at    (5+0,0)        {\scriptsize };
\node[mi] (b) at    (5+0,-1)       {\scriptsize $a_3$};
\node[mi] (b1) at   (5+1,-2)       {\scriptsize $a_3$};
\node[mi] (c) at    (5+0,-2)       {\scriptsize $a_1$};
\node[mi] (d) at    (5+0,-3)       {\scriptsize $a_2$};
\node[mi] (e) at    (5+0,-4)       {\scriptsize $a_1$};
\node[mi] (e1) at   (5+-1.5,-5)      {\scriptsize $a_4$};
\node[mi] (e11) at  (5+-2,-6)      {\scriptsize };
\node[mi] (e12) at  (5+-1,-6)      {\scriptsize $a_4$};
\node[bxl] (f) at   (5+0,-5)       {\scriptsize $a_1$};
\node[mi] (g) at    (5+0,-6)       {\scriptsize $a_2$};
\node[mi] (g1) at   (5+1.5,-7)       {\scriptsize $a_5$};
\node[mi] (g11) at  (5+1.5+.5,-8)       {\scriptsize $a_6$};
\node[mi] (g12) at  (5+1-.5+.5,-8)       {\scriptsize $a_5$};
\node[mi] (g111) at (5+1.5+.5*2,-9)       {\scriptsize };
\node[mi] (g112) at (5+1.5,-9)       {\scriptsize $a_6$};
\node[bxl] (h) at   (5+0,-7)       {\scriptsize $a_2$};
\node[mi] (i) at    (5+0,-8)       {\scriptsize $a_7$};
\node[mi] (j) at    (5+0.75,-9)    {\scriptsize $a_7$};
\node[mi] (j1) at   (5+-0.75,-9)   {\scriptsize };


\draw (a) -- (b) -- (c) -- (d) -- (e) -- (f) -- (g) -- (h) -- (i) -- (j);
\draw (i) -- (j1);
\draw (b) -- (b1);
\draw (e) -- (e1);
\draw (e11) -- (e1);
\draw (e12) -- (e1);
\draw (g) -- (g1);
\draw (g11) -- (g1);
\draw (g12) -- (g1);
\draw (g11) -- (g111);
\draw (g11) -- (g112);
\draw (c) to[bend left] (f);
\draw (d) to[bend right] (h);


\draw [red, ultra thick] plot [smooth cycle] coordinates {(5+0,0.5) (5+0.45,-0.1) (5+0.5,-3.7) (5+0.1,-4.4) (5+-1.3,-5.4)(5-2+0.3,-6.2) (5-2,-6.4)(5-2-0.3,-6)(5+-1.9,-4.9)(5+-0.5,-4) (5+-0.45,-0.1)};
\draw [blue, ultra thick] plot [smooth cycle] coordinates {(5+0,-4.6)(5+0.45,-5)(5.6,-6)(6.8,-6.8)(7.8,-9)(7.5,-9.4)(6.5,-7.5)(4.6,-6)(4.55,-5)};
\draw [green, ultra thick] plot [smooth cycle] coordinates {(5,-6.6)(5.45,-7)(5.5,-8)(4.25,-9.5)(3.75,-9)(4.5,-8)(4.55,-7)};
\draw[orange, ultra thick] (b1) circle (0.3cm);
\draw[olive, ultra thick] (g112) circle (0.3cm);
\draw[teal, ultra thick] (e12) circle (0.3cm);
\draw[violet, ultra thick] (g12) circle (0.3cm);
\draw[purple, ultra thick] (j) circle (0.3cm);

\end{tikzpicture}
\caption{\small Example of a spinal caterpillar network and its decomposition into path components.}\label{fig:spinalcater}

\end{figure}

Now we will introduce a result analogous to Proposition~\ref{prop:wordsbessel}, that gives an explicit formula for $|C_{n,k}^2/{\sim}|$.

\begin{prop} \label{prop:wordsbessel2} 
For \(n\geq 1\) and \(k\leq n\) we have
\[
|C_{n,k}^2/{\sim}| = {n+2k\choose n-k}\frac{(2k)!}{2^kk!}.
\]
\end{prop}

\begin{proof}
Let $w \in C_{n,k}^2$. As in the proof of Proposition~\ref{prop:wordsbessel}, we begin 
by removing the fixed suffix $a_{k+1}\cdots a_n$, leaving a word of length $n+2k$ 
containing three occurrences of each $a_i$ for $1\le i\le k$ and one occurrence of 
each $a_j$ for $k+1\le j\le n$.

By definition of $C_{n,k}^2$, between the second and third occurrences of any $a_i$ 
only letters from $\{a_{k+1},\ldots,a_n\}$ may appear. Once these singleton letters 
are removed, the second and third occurrences of each $a_i$ become consecutive in the 
resulting word. Compressing each such consecutive pair into a single block, exactly as 
in the previous proof, yields a word $w^\ast$ of length $2k$, where each letter $a_i$ 
appears once as a singleton (its first occurrence) and once as a block.

By the same bijection argument as before, $w^\ast$ determines a partition of the $2k$ 
positions $\{1,\ldots,2k\}$ into $k$ unordered pairs, and the number of such 
partitions is
\[
    \binom{2k}{2k}\frac{(2k)!}{2^k k!} = \frac{(2k)!}{2^k k!}.
\]
It remains to place the $n-k$ first appearances of the letters $a_{k+1},\ldots,a_n$ among the $n+2k$ positions of the word prior to compression. Since we work modulo $\sim$, only the 
choice of positions matters, not the assignment of letters to them, giving 
$\binom{n+2k}{n-k}$ choices.

By the product principle,
\[
    |C_{n,k}^2/{\sim}| \;=\; \binom{n+2k}{n-k}\cdot\frac{(2k)!}{2^k k!}. 
\]
\end{proof}

As a direct consequence of Propositions  \ref{prop:caterpillar_bijection} and \ref{prop:wordsbessel2}, we obtain the following corollary.

\begin{coro}
      For $n\geq 1$ and $k\leq n-1$ we have

    $$|\NLSCTC_{n,k}| = \binom{n-1+2k}{n-1-k}\cdot\frac{(2k)!}{2^k k!}. 
   $$
\end{coro}

\section{Generating functions for counting spinal networks}\label{sec:gf}

In this section we associate to spinal tree-child networks a class of marked trees that admit an explicit combinatorial specification. This specification is then translated into bivariate generating functions, which allow us to derive explicit counting formulas for the number of spinal tree-child networks in both the labeled and unlabeled settings.


\subsection{Marked trees and their relation with spinal tree-child networks}

In this subsection we construct a class of \emph{marked trees} associated to a certain subclass of spinal tree-child networks, and we establish a precise correspondence between this marked class and the original class $\STC$ of spinal tree-child networks. 



Let $N \in \STC_{n,k}$ be a spinal tree-child network with spine $(u_0,u_1,\ldots,u_\ell)$. The spinal structure induces a natural linear ordering of the reticulation vertices, which we label by $\{r_i\}_{i=1}^k$ (notice that, for clarity, we now use a different notation for reticulations than the one used in the previous section), according to their occurrence along the spine from the root to the terminal leaf. Also, to ease notation, we label the root $u_0$ with $r_0$ and the terminal leaf $u_\ell$ with $r_{k+1}$. Also, for each $i=0,\ldots,k+1$, we let $\ell_i$ be the index such that $u_{\ell_i}$ is labeled with $r_i$. Finally, for each $i=1,\ldots,k+1$ we (temporarily) label with $s_i$ the node $u_{\ell_i-1}$, that is, the parent along the spine of the node labeled with $r_i$. We shall hereafter identify the newly labeled nodes with their labels.

One remark has to be made regarding the last reticulation vertex $r_k$. There are two cases. The first one corresponds to the case where $r_{k+1}$, the terminal leaf of the spine, is the single child of $r_k$, the last reticulation vertex of the spine. In this case, $r_k=s_{k+1}$. The second case corresponds to the case where $r_{k+1}$ has a sibling, which must be a leaf. In this case, $s_{k+1}$ is a tree node and $r_k\neq s_{k+1}$. We shall refer to the first case as the \emph{non-cherry} case, and to the second one as the \emph{cherry} case. 
We shall indicate with $\STCch_{n,k}$ and $\STCnc_{n,k}$ the subclasses of $\STC_{n,k}$ consisting of networks falling in the cherry and non-cherry cases, respectively. 

Our strategy will be to count these two subclasses of networks separately in order to obtain the global count.

\subsubsection*{The cherry case}

Let $N$ be a spinal tree-child network in $\STCch_{n,k}$. According to the labeling described above, we have that all the nodes $s_1,\ldots,s_{k+1},r_0,r_1,\ldots,r_{k+1}$ are pairwise different.

We perform on $N$ the following operation: We remove the arcs $s_jr_j$ for $j=1,\ldots,k+1$, and we remove the leaf $r_{k+1}$ (which has become an isolated node), keeping the labels of the $n-1$ leaves and the nodes $r_0,r_1,\ldots,r_k,s_1,\dots,s_{k+1}$, which have become \emph{elementary} nodes (that is, nodes with indegree and outdegree equal to one). Notice that the single child of each node $s_i$ is a leaf, by the same arguments given in the proof of Proposition~\ref{prop:encod}.
We refer to $T$ as the \emph{marked tree} associated to $N$.

With the removal of the mentioned arcs, the spine of $N$ is broken into $k+1$ directed paths $P_0,\dots,P_k$, where each $P_i$ connects $r_i$ to $s_{i+1}$. We shall refer to these paths as the \emph{spinal subpaths} of $T$. These paths cover all the internal nodes of $T$, and the only nodes that do not lie on any spinal subpath are the $n-1$ leaves. Each path starts and ends in an elementary node, and each intermediate node of the path has exactly two children, one of which is the next node in the path and the other one is either a leaf or the first node of another spinal subpath. Note also that none of these paths can be trivial (i.e., consisting of a single node), since otherwise in $N$ we would have two consecutive reticulations, which is not possible. Also, the single child of each elementary node $s_i$ must be a leaf.

It is obvious that $N$ can be reconstructed from $T$ and the label of the removed leaf $r_{k+1}$. Indeed, one only needs to restore the arcs from $s_i$ to $r_i$ for $i=1,\ldots,k+1$, and to restore the leaf $r_{k+1}$ as a child of $s_{k+1}$.

Notice also that the labels $s_i$ can be safely removed from $T$ without losing any information. Indeed, if we remove these labels, for each unlabeled elementary node, by following the tree upwards, we eventually reach a node labeled with some $r_i$, and we can infer that the label of the elementary node was $s_{i+1}$. From now on, we shall forget about these labels and consider $T$ as a graph with $n-1$ labeled leaves, $k+1$ labeled elementary nodes (one of which is the root $r_0$), and $k+1$ unlabeled elementary nodes. We let $\MT_{n-1,k+1}$ denote the class of all such marked trees obtained from spinal tree-child networks in $\STC_{n,k}$ through the above construction.

One important remark has to be made regarding the class $\MT_{n-1,k+1}$. The construction of a marked tree requires the choice of one of the two possible spines in the original network $N$, and the removal of the terminal leaf of that spine. Hence, each network $N$ gives rise to two marked trees, one per choice of spine. Moreover, the terminal leaf that is removed can carry any label in $[n]$, so the leaf set of the resulting marked tree is a subset of $[n]$ of size $n-1$, not necessarily $[n-1]$ itself.
We temporarily keep the original leaf labels, so that for each $j\in[n]$, we write $\MT_{n-1,k+1}(j)$ for the class of marked trees whose leaf set is $[n]\setminus\{j\}$. The map that sends a pair $(N,\sigma)$, where $N\in\STCch_{n,k}$ and $\sigma$ is a spine of $N$ with terminal leaf labeled $j$, to the corresponding marked tree in $\MT_{n-1,k+1}(j)$ is a bijection. Therefore
$$2\,|\STCch_{n,k}| \;=\; \sum_{j=1}^{n}|\MT_{n-1,k+1}([n]\setminus\{j\})|.$$
By symmetry of the labels, all the classes $\MT_{n-1,k+1}(j)$ have the same cardinality, which we denote simply by $|\MT_{n-1,k+1}|$. Hence
$$|\STCch_{n,k}| \;=\; \frac{n}{2}\,|\MT_{n-1,k+1}|.$$

We illustrate in Figure~\ref{fig:marked_alternative} the construction of the marked version of a spinal tree-child network in the cherry case. Notice that the leaf that has been removed in the construction is the one labeled with $5$; we have done so to emphasize that $T$ has $n-1$ leaves, but the removed leaf needs not be the one labeled with $n$ in the original network. The first subfigure depicts the original spinal tree-child network, the second one shows the labels of the nodes $r_i$ and $s_i$, the third one shows the marked version of the network, and the fourth one shows the same marked version but without the labels of the nodes $s_i$. The spinal subpaths are indicated with thicker lines.

We can also give an intrinsic description of the class $\MT_{n-1,k+1}$, without referring to the original spinal tree-child networks. Indeed, the class $\MT_{n-1,k+1}$ consists of rooted trees with $n-1$ labeled leaves (say by the set $\{1,\dots,n-1\}$), $k+1$ labeled elementary nodes (say by the set $\{r_0,\dots,r_k\}$), and $k+1$ unlabeled elementary nodes, such that the following conditions hold:
\begin{itemize}
    \item The root is labeled with $r_0$, and in particular it is an elementary node.
    \item Following the tree upwards from any unlabeled elementary node, the first elementary node that is encountered is labeled, say with label $r_i$. We shall refer to the starting unlabeled elementary node as $s_{i+1}$, and we can define $P_i$ as the path from $r_i$ to $s_{i+1}$ (with both ends included). 
    \item The paths $P_0,\dots,P_k$ defined in the previous item are pairwise disjoint, and they cover all the internal nodes of the tree.
    \item Each intermediate node of each path $P_i$ has exactly one child out of the path, which is either a leaf or the first node of another path $P_j$ with $j>i$, and in case that this intermediate node is the last intermediate node of $P_i$, then its child that does not lie on $P_i$ is a leaf.
\end{itemize}

Indeed, given a tree $T$ satisfying the above conditions, we can reconstruct a spinal tree-child network $N$ by adding an isolated node labeled with $r_{k+1}$, and the arcs from $s_i$ to $r_i$ for $i=1,\ldots,k+1$. Notice that the added node is now a leaf and we label it with $n$. The condition that if a node in a path $P_i$ has a child that is the first node of another path $P_j$, then $j>i$, ensures that the resulting graph is acyclic. The condition that each internal node of $T$ lies on some path $P_i$, together with the condition the last intermediate node of each path $P_i$ has a leaf as a child, ensures that the resulting network is tree-child.
The condition that the paths cover the internal nodes of $T$ ensures that the resulting graph admits a spine, which will be formed by glueing together the paths $P_0,\dots,P_k$ using the restored arcs, including the one that ends at the added leaf. 

Furthermore, this intrinsic description of $\MT_{n-1,k+1}$ can also be reformulated in a recursive way. Indeed, a tree $T$ in $\MT_{n-1,k+1}$ can be described as follows. Begin with a path $P_0$ starting at the root $r_0$, and for each intermediate node of the path, attach to it another child, which is either a leaf or the first node of another path $P_i$ rooted at some $r_i$ with $i>0$ (with the condition that a leaf must be attached to the last intermediate node of the path). This path $P_i$ rooted at $r_i$ is analogously described as a path whose intermediate nodes have attached to them either leaves or the first node of another path $P_j$ rooted at some $r_j$ with $j>i$, and so on. 

\begin{figure}[h]
\centering
\begin{tikzpicture}[scale=0.6]

    \node[mi] (a) at    (0,0)        {};
    \node[mi] (b) at    (0,-1)       {};
    \node[mi] (b1) at   (1,-2)       {\small  1};
    \node[mi] (c) at    (0,-2)       {};
    \node[mi] (d) at    (0,-3)       {};
    \node[mi] (e) at    (0,-4)       {};
    \node[mi] (e1) at   (-1,-5)      {\small  2};
    \node[bxl] (f) at   (0,-5)       {};
    \node[mi] (g) at    (0,-6)       {};
    \node[mi] (h1) at   (1,-8)       {\small  3};
    \node[mi] (h) at    (0,-7)       {};
    \node[bxl] (i) at   (0,-8)       {};
    \node[mi] (j) at    (0,-9)       {};
    \node[bxl] (k) at   (0,-10)       {};
    \node[mi] (k1) at   (0+1,-10)       {\small 4};
    \node[mi] (l) at    (0,-11)       {};
    \node[mi] (m) at    (0.75,-12)    {\small 6};
    \node[mi] (m1) at   (-0.75,-12)   {\small 5};

    \draw (a) -- (b) -- (c) -- (d) -- (e) -- (f) -- (g) -- (h) -- (i) -- (j)--(k) --(l)--(m);
    \draw (l) -- (m1);
    \draw (b) -- (b1);
    \draw (e) -- (e1);
    \draw (h) -- (h1);
    \draw (j) -- (k1);
    \draw (c) to[bend left] (f);
    \draw (g) to[bend right] (k);
    \draw (d) to[bend right] (i);
\end{tikzpicture}
\qquad
\begin{tikzpicture}[scale=0.6]

    \node[bxl] (a) at    (0,0)        {\small$r_0$};
    \node[mi] (b) at    (0,-1)       {};
    \node[mi] (b1) at   (1,-2)       {\small  1};
    \node[mi] (c) at    (0,-2)       {};
    \node[mi] (d) at    (0,-3)       {};
    \node[mi] (e) at    (0,-4)       {\small $s_1$};
    \node[mi] (e1) at   (-1,-5)      {\small  2};
    \node[bxl] (f) at   (0,-5)       {\small $r_1$};
    \node[mi] (g) at    (0,-6)       {};
    \node[mi] (h1) at   (1,-8)       {\small  3};
    \node[mi] (h) at    (0,-7)       {\small $s_2$};
    \node[bxl] (i) at   (0,-8)       {\small $r_2$};
    \node[mi] (j) at    (0,-9)       {\small $s_3$};
    \node[bxl] (k) at   (0,-10)       {\small $r_3$};
    \node[mi] (k1) at   (0+1,-10)       {\small 4};
    \node[mi] (l) at    (0,-11)       {\small $s_4$};
    \node[mi] (m) at    (0.75,-12)    {\small 6};
    \node[mi] (m1) at   (-0.75,-12)   {\small $r_4$};

    \draw (a) -- (b) -- (c) -- (d) -- (e) -- (f) -- (g) -- (h) -- (i) -- (j) -- (k) --(l)--(m);
    \draw (l) -- (m1);
    \draw (b) -- (b1);
    \draw (e) -- (e1);
    \draw (h) -- (h1);
    \draw (j) -- (k1);
    \draw (c) to[bend left] (f);
    \draw (g) to[bend right] (k);
    \draw (d) to[bend right] (i);
\end{tikzpicture}
\qquad
\begin{tikzpicture}[scale=0.6]

    \node[bxl] (a) at    (0,0)        {\small$r_0$};
    \node[mi] (b) at    (0,-1)       {};
    \node[mi] (b1) at   (1,-2)       {\small  1};
    \node[mi] (c) at    (0,-2)       {};
    \node[mi] (d) at    (0,-3)       {};
    \node[mi] (e) at    (0,-4)       {\small $s_1$};
    \node[mi] (e1) at   (-1,-5)      {\small  2};
    \node[bxl] (f) at   (0,-5)       {\small $r_1$};
    \node[mi] (g) at    (0,-6)       {};
    \node[mi] (h1) at   (1,-8)       {\small  3};
    \node[mi] (h) at    (0,-7)       {\small $s_2$};
    \node[bxl] (i) at   (0,-8)       {\small $r_2$};
    \node[mi] (j) at    (0,-9)       {\small $s_3$};
    \node[bxl] (k) at   (0,-10)       {\small $r_3$};
    \node[mi] (k1) at   (0+1,-10)       {\small 4};
    \node[mi] (l) at    (0,-11)       {\small $s_4$};
    \node[mi] (m) at    (0.75,-12)    {\small 6};

    \draw (a) -- (b) -- (c) -- (d) -- (e);
    \draw (f) -- (g) -- (h);
    \draw (i) -- (j);
    \draw (k) --(l)--(m);
    \draw (b) -- (b1);
    \draw (e) -- (e1);
    \draw (h) -- (h1);
    \draw (j) -- (k1);
    \draw (c) to[bend left] (f);
    \draw (g) to[bend right] (k);
    \draw (d) to[bend right] (i);
\end{tikzpicture}
\qquad
\begin{tikzpicture}[scale=0.6]

    \node[bxl] (a) at    (0,0)        {\small$r_0$};
    \node[mi] (b) at    (0,-1)       {};
    \node[mi] (b1) at   (1,-2)       {\small  1};
    \node[mi] (c) at    (0,-2)       {};
    \node[mi] (d) at    (0,-3)       {};
    \node[mi] (e) at    (0,-4)       {};
    \node[mi] (e1) at   (-1,-5)      {\small  2};
    \node[bxl] (f) at   (0,-5)       {\small $r_1$};
    \node[mi] (g) at    (0,-6)       {};
    \node[mi] (h1) at   (1,-8)       {\small  3};
    \node[mi] (h) at    (0,-7)       {};
    \node[bxl] (i) at   (0,-8)       {\small $r_2$};
    \node[mi] (j) at    (0,-9)       {};
    \node[bxl] (k) at   (0,-10)       {\small $r_3$};
    \node[mi] (k1) at   (0+1,-10)       {\small 4};
    \node[mi] (l) at    (0,-11)       {};
    \node[mi] (m) at    (0.75,-12)    {\small 6};

    \draw[very thick] (a) -- (b) -- (c) -- (d) -- (e);
    \draw[very thick] (f) -- (g) -- (h);
    \draw[very thick] (i) -- (j);
    \draw[very thick] (k) --(l);
    \draw (l)--(m);
    \draw (b) -- (b1);
    \draw (e) -- (e1);
    \draw (h) -- (h1);
    \draw (j) -- (k1);
    \draw (c) to[bend left] (f);
    \draw (g) to[bend right] (k);
    \draw (d) to[bend right] (i);
\end{tikzpicture}
\caption{Depiction of the process of obtaining the marked version of a spinal tree-child network.\label{fig:marked_alternative}}
\end{figure}
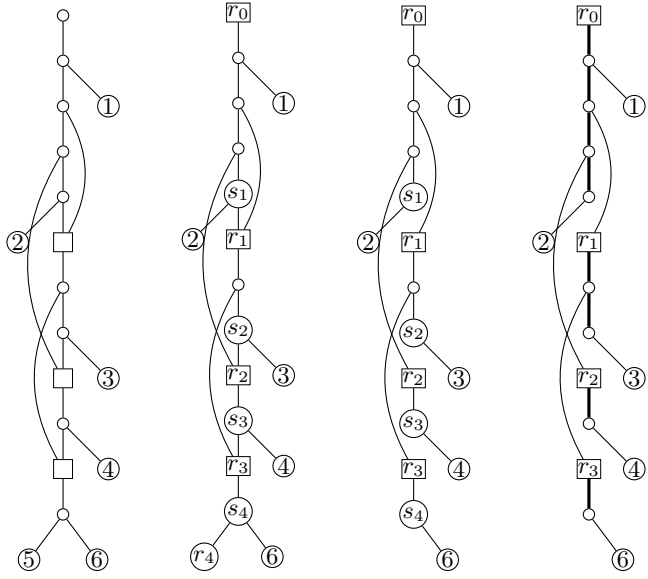

\subsubsection*{The non-cherry case}

Now we consider the case where $N$ is a spinal tree-child network in $\STC_{n,k}$ such that the terminal leaf has no sibling, and is the single child of the last reticulation vertex $r_k$. We shall proceed by associating to $N\in\STCnc_{n,k}$ another spinal tree-child network $N' \in \STCch_{n,k-1}$. The reticulation vertex $r_k$ has two parents, one of which is $s_k$, and the other one is a node $p$ that lies in the spine. Notice that the other child of $p$ (along the spinal path) must be a tree node, since otherwise it would have two reticulation children, and that the other child of $s_k$ must be a leaf. After removing the arc from $p$ to $r_k$, both nodes become elementary. Each of them can then be deleted, linking its single parent directly to its single child. See Figure~\ref{fig:marked_alternative_degenerate} for a depiction of this process.

The resulting graph $N'$ is a spinal tree-child network in $\STCch_{n,k-1}$. Indeed, it is clear that it admits a spine, and the fact that it is tree-child follows from the remark that the other child of $p$ is a tree node. Finally it falls in the cherry case since  now the leaf originally on the spine has a sibling, the child of $s_{k}$ outside the spine. 

Going the other way, given a spinal tree-child network in the cherry case $N' \in \STCch_{n,k-1}$, we can construct a spinal tree-child network corresponding to the non-cherry case $N\in\STCnc_{n,k}$. 
Indeed, from the construction above, it follows that the process amounts to: (1) choosing two suitable arcs on the spine of $N’$ (see below for the restrictions on the admissible arcs); (2) subdividing each of them with a new node, which will be identified as $p$ and $r_k$; and (3) connecting $p$ to $r_k$.
The arc on which $r_k$ is placed must be one of the two arcs forming the cherry at the end of $N’$. As for the placement of $p$, the chosen arc must lie on the spine and its terminal node must be a tree node. Hence, there are as many choices as tree nodes, namely $n+(k-1)-1 = n+k-2$ (see Lemma~\ref{lem:lengthspine}). 
Therefore, $N$ can be reconstructed from $N’$ in $2(n+k-2)$ ways, and consequently
\[
|\STCnc_{n,k}| = 2(n+k-2)\,|\STCch_{n,k-1}|.
\]


\subsubsection*{Counting networks in the labeled and unlabeled setting}

In the previous subsections we have found a relation between the number of labeled spinal networks, both in the cherry and the non-cherry case. These results can be summarized in the following result.

\begin{prop}\label{prop:marked}
For all $n\geq 1$ and $k\geq 0$,
$$|\STC_{n,k}| = \frac{n}{2}\,|\MT_{n-1,k+1}| + n(n+k-2)\,|\MT_{n-1,k}|.$$
\end{prop}

We can adapt the counting given above for labeled spinal networks to unlabeled networks, both in the cherry and in the non-cherry case. 

In the first place, the characterization of the class $\MT_{n-1,k+1}$ as the marked trees obtained from networks in $\STCch_{n,k}$, as well as both the intrinsic and recursive descriptions given in the subsection dedicated to the cherry case can be translated \emph{mutatis mutandis} to the unlabeled case. We shall refer to this class of marked trees as $\NLMT$, and notice that trees in this class have their leaves unlabeled, but some of their elementary nodes are labeled. Also, we will indicate by $\NLSTCch$ and $\NLSTCnc$ the unlabeled spinal tree-child networks corresponding respectively to the cherry and non-cherry case.

Some more caution has to be taken in order to relate the number of networks in $\NLSTCch_{n,k}$ and the number of trees in $\NLMT_{n-1,k+1}$. We recall that, when counting how many different networks can be obtained from each marked tree, in the labeled case we had to keep in mind that each network has two different spines, corresponding to the two (distinguishable) leaves of the final cherry. However, in the unlabeled setting, a spinal network in the cherry case has a single spine, since both leaves in the cherry are indistinguishable. Also, the counting we had to do in order to deal with which of the integers in $[n]$ was removed to construct the tree also loses its meaning, since all the leaves are indistinguishable. In other words, given an unlabeled marked tree, a single spinal network ending in a cherry can be reconstructed. Hence, we get that
$$|\NLSTCch_{n,k}|=|\NLMT_{n-1,k+1}|.$$

As for the non-cherry case, we have to count how many such networks can be obtained from a given network in the cherry case. Following the notations given above in the labeled case, the two choices of where to put the node $r_k$ reduce to a single one (since the leaves in the cherry are interchangeable).
As for the choice of where to allocate the node $p$, the $n+k-2$ possibilities in the labeled case remain pairwise distinct, since these choices are determined by their position along the spine rather than by leaf labels; consequently, the enumeration is unchanged.
 Hence we get that
$$|\NLSTCnc_{n,k}|=(n+k-2)|\NLSTCch_{n,k-1}|.$$

These computations can be summarized in the following result.

\begin{prop}\label{prop:marked_unlabeled}
For all $n\geq 1$ and $k\geq 0$,
$$|\NLSTC_{n,k}| = |\NLMT_{n-1,k+1}| + (n+k-2) |\NLMT_{n-1,k}|.$$
\end{prop}

\begin{figure}[h]
\centering
\begin{tikzpicture}[scale=0.6]

    \node[mi] (a) at    (0,0)        {};
    \node[mi] (b) at    (0,-1)       {};
    \node[mi] (b1) at   (1,-2)       {\small  1};
    \node[mi] (c) at    (0,-2)       {};
    \node[mi] (d) at    (0,-3)       {};
    \node[mi] (e) at    (0,-4)       {};
    \node[mi] (e1) at   (-1,-5)      {\small  2};
    \node[bxl] (f) at   (0,-5)       {};
    \node[mi] (g) at    (0,-6)       {\small $p$};
    \node[mi] (h1) at   (1,-8)       {\small  3};
    \node[mi] (h) at    (0,-7)       {};
    \node[bxl] (i) at   (0,-8)       {};
    \node[mi] (j) at    (0,-9)       {\small $s_3$};
    \node[bxl] (k) at   (0,-10)       {\small $r_3$};
    \node[mi] (k1) at   (0+1,-10)       {\small 4};
    \node[mi] (l) at    (0,-11)       {\small 5};

    \draw (a) -- (b) -- (c) -- (d) -- (e) -- (f) -- (g) -- (h) -- (i) -- (j)--(k) --(l);
    \draw (b) -- (b1);
    \draw (e) -- (e1);
    \draw (h) -- (h1);
    \draw (j) -- (k1);
    \draw (c) to[bend left] (f);
    \draw (g) to[bend right] (k);
    \draw (d) to[bend right] (i);
\end{tikzpicture}
\qquad
\begin{tikzpicture}[scale=0.6]

    \node[mi] (a) at    (0,0)        {};
    \node[mi] (b) at    (0,-1)       {};
    \node[mi] (b1) at   (1,-2)       {\small  1};
    \node[mi] (c) at    (0,-2)       {};
    \node[mi] (d) at    (0,-3)       {};
    \node[mi] (e) at    (0,-4)       {};
    \node[mi] (e1) at   (-1,-5)      {\small  2};
    \node[bxl] (f) at   (0,-5)       {};
    \node[mi] (g) at    (0,-6)       {\small $p$};
    \node[mi] (h1) at   (1,-8)       {\small  3};
    \node[mi] (h) at    (0,-7)       {};
    \node[bxl] (i) at   (0,-8)       {};
    \node[mi] (j) at    (0,-9)       {\small $s_3$};
    \node[bxl] (k) at   (0,-10)       {\small $r_3$};
    \node[mi] (k1) at   (0+1,-10)       {\small 4};
    \node[mi] (l) at    (0,-11)       {\small 5};

    \draw (a) -- (b) -- (c) -- (d) -- (e) -- (f) -- (g) -- (h) -- (i) -- (j)--(k) --(l);
    \draw (b) -- (b1);
    \draw (e) -- (e1);
    \draw (h) -- (h1);
    \draw (j) -- (k1);
    \draw (c) to[bend left] (f);
    \draw (d) to[bend right] (i);
\end{tikzpicture}
\qquad
\begin{tikzpicture}[scale=0.6]

    \node[mi] (a) at    (0,0)        {};
    \node[mi] (b) at    (0,-1)       {};
    \node[mi] (b1) at   (1,-2)       {\small  1};
    \node[mi] (c) at    (0,-2)       {};
    \node[mi] (d) at    (0,-3)       {};
    \node[mi] (e) at    (0,-4)       {};
    \node[mi] (e1) at   (-1,-5)      {\small  2};
    \node[bxl] (f) at   (0,-5)       {};
    \node[mi] (h1) at   (1,-8)       {\small  3};
    \node[mi] (h) at    (0,-7)       {};
    \node[bxl] (i) at   (0,-8)       {};
    \node[mi] (j) at    (0,-9)       {};
    \node[mi] (k1) at   (0+1,-10)       {\small 4};
    \node[mi] (l) at    (0,-11)       {\small 5};

    \draw (a) -- (b) -- (c) -- (d) -- (e) -- (f) -- (h) -- (i) -- (j) --(l);
    \draw (b) -- (b1);
    \draw (e) -- (e1);
    \draw (h) -- (h1);
    \draw (j) -- (k1);
    \draw (c) to[bend left] (f);
    \draw (d) to[bend right] (i);
\end{tikzpicture}
\caption{Depiction of the process of obtaining a cherry spinal tree-child network from a non-cherry one.\label{fig:marked_alternative_degenerate}}
\end{figure}
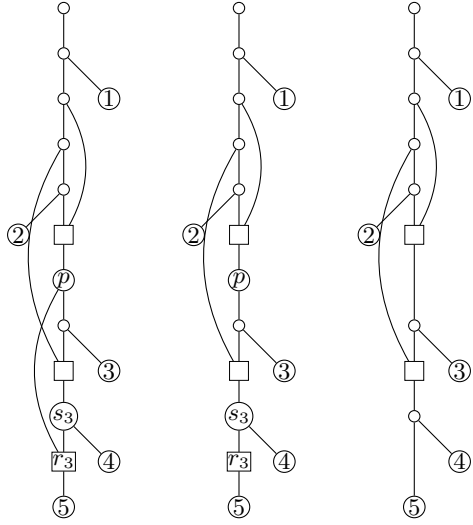

\subsection{Recursive specification of marked trees}

In this subsection we give a recursive specification for the class $\MT$ of marked trees. We shall use combinatorial classes and the symbolic method, as presented in \cite{flajolet2009analytic}. Apart from the standard constructors of disjoint union, product and sequence, we will also use the boxed product \cite[\S II.6]{flajolet2009analytic}, which is a key construction in our specifications. 

We will use combinatorial classes that are biweighted, corresponding to the number of leaves and reticulations, and bilabeled, corresponding to the labels of the leaves and of the labeled elementary nodes (equivalently, the sequence of the reticulations along the spine of the reconstructed spinal network). In order to distinguish them, the labels of the leaves will be  integers, while those of elementary nodes will be symbols $r_i$, indexed by integers.

Let $\mathcal{L}$ be the atomic class formed by a single labeled leaf, and $\mathcal{R}$ the atomic class formed by a single labeled elementary node. 

One remark needs to be made regarding the interpretation of the boxed product when applied to biweighted classes, since the restriction on the labels can refer to any of the two sets of labels. In our setting, the boxed product $\cA\boxprod\cB$ will indicate that the smallest label of the $\cR$-atoms in the resulting structure lies in the $\cA$-component. In particular, the elements in the boxed product $\cR\boxprod \cB$ are constructed by taking the pair formed by the single atom in $\cR$, which gets the smallest label, namely $r_0$, and an element of $\cB$ where each atom labeled with say $r_i$ is relabeled to $r_{i+1}$, for $i\geq 0$.

\begin{prop}\label{prop:especif_marked_trees}
The class $\MT$ admits the following recursive specification:
$$\MT = \mathcal{R} \boxprod \big(\Seq(\MT \sqcup \mathcal{L}) \star \mathcal{L}\big).$$
\end{prop}

\begin{proof}
    From the recursive definition of $\MT$ given in the previous subsection, an element of $\MT$ can be viewed as a path starting at the root, where at each vertex different from the root we attach either a leaf or another marked tree, with the additional requirement that a leaf is attached to the final vertex.

    Since the relevant information of this construction is the starting node of the path, which is an atom in $\cR$ with its corresponding label, and the sequence of either leaves or marked trees we hang along the path (with the restriction that the last one must be a leaf), which are elements either in $\MT$ or $\cL$, this gives us the recursive specification stated in the proposition. Note that the fact that we use the boxed product ensures that, whenever we consider a path starting at a labeled elementary node, the roots of the attached marked trees are labeled with larger labels, as required by the definition of $\MT$.
\end{proof}

This result can be adapted to the case where the leaves are unlabeled.

\begin{prop}\label{prop:especif_non_labeled}
The class $\NLMT$ admits the recursive specification
\[
\NLMT \;=\; \cR \boxprod \Big(\Seq(\NLMT \sqcup \cL)\times \cL\Big).
\]
\end{prop}

\begin{proof}
This result follows by adapting Proposition~\ref{prop:especif_marked_trees} to a setting in which leaves are unlabeled, while elementary nodes retain their labels. More precisely, the labeled product inside the parenthesis in Proposition~\ref{prop:especif_marked_trees} is substituted by a simple product. The boxed product, however, must be retained, since it accounts for the labels of elementary nodes. Notice that since only elementary nodes are labeled, there is no ambiguity as to the set of labels to which the restriction imposed by the boxed product applies. 
\end{proof}

\subsection{Counting spinal tree-child networks by generating functions}

We now translate the recursive specification of Proposition~\ref{prop:especif_marked_trees} into generating functions in the labeled setting. 
We use two variables: $x$ marks leaves and $z$ marks labeled elementary vertices. Let
\[
S(x,z)=\sum_{n,k\ge 0}\frac{s_{n,k}}{n!\,k!}\,x^n z^k
\]
be the bivariate exponential generating function of  $\MT$, where $n$ counts  (labeled) leaves  and $k$ counts (labeled) elementary vertices.



\begin{prop}\label{prop:s_coef}
The coefficients $s_{n,k}$ satisfy the following explicit formulation:
    $$s_{n,k}=\left\{ \begin{array}{ccl}
   0 & if  & n<k \\
   \displaystyle\frac{n!\cdot (n+k-2)!}{2^{k-1}(k-1)!(n-k)!}& if  & n\geq k
\end{array}\right.$$
\end{prop}

\begin{proof}
From Proposition~\ref{prop:especif_marked_trees}, we apply \cite[Theorem II.5]{flajolet2009analytic} with respect to the variable $z$, treating $x$ as a parameter. Then

\begin{equation}\label{eqSint}
S(x,z)=\int_{0}^z \frac{x}{1-(x+S(x,t))} dt.
\end{equation}

By differentiating under the integral sign (all derivatives are with respect to $z$)
\begin{equation}\label{eqS'}
    S'(x,z)=\frac{x}{1-x-S(x,z)}
\end{equation}

Equation~(\ref{eqSint}) gives $S(x,0)=0$ and equation~(\ref{eqS'}) gives $S'(x,0)=x/(1-x)=x+x^2+x^3+\cdots$.

Note that $S'(x,0)=\sum_{n\geq 1} \frac{s_{n,1}}{n!} x^n$, so we get that $s_{n,1}=n!$, which one can easily check is what it should be (spinal networks with just the trivial reticulation are a path of leaves).

Equation~(\ref{eqS'}) is a first-order separable differential equation, and can therefore be solved explicitly. The solution is given by

  $$S(x,z) = 1-x-\sqrt{(1-x)^2-2xz}$$

Differentiating this expression with respect to $z$ and evaluating at $z = 0$, we obtain

$$\frac{d^mS}{dz^m}(x,0)  = S^{(m)}(x,0) = (2m-3)!!\cdot \frac{x^m}{(1-x)^{2m-1}}$$

This expression can be rewritten in terms of a generating function as

\begin{equation}\label{eqSm}
 S^{(m)}(x,0)  = \sum_{n\geq m} (2m-3)!!\cdot {n+m-2\choose 2m-2}x^n\end{equation}

The coefficients $s_{n,k}$ can be obtained from equation~(\ref{eqSm}) using the identity

$$S^{(k)}(x,0) = \sum_{n\geq 1} \frac{s_{n,k}}{n!} x^n.$$

More precisely, for $n\ge k$ we obtain
\[
\frac{s_{n,k}}{n!}=(2k-3)!!\binom{n+k-2}{2k-2}.
\]
Using $(2k-3)!!=\frac{(2k-2)!}{2^{k-1}(k-1)!}$ and
$\binom{n+k-2}{2k-2}=\frac{(n+k-2)!}{(2k-2)!(n-k)!}$, we get
\[
s_{n,k}=\frac{n!(n+k-2)!}{2^{k-1}(k-1)!(n-k)!},
\]
and $s_{n,k}=0$ for $n<k$.
\end{proof}

By combining Proposition~\ref{prop:marked} with Proposition~\ref{prop:s_coef}, this yields Theorem \ref{teo:stccount}, the explicit enumeration formula for the class $\STC_{n,k}$.

\begin{prop}\label{prop:stc_formula_gf}
For $n>1$ we have
\[
|\STC_{n,k}| \;=\; \frac{n!(n+k-2)!(n+3k-1)}{2^{k+1}k!(n-k-1)!}.
\]
\end{prop}

\begin{proof}
By Proposition~\ref{prop:marked},
\[
|\STC_{n,k}|=\frac{n}{2}\,s_{n-1,k+1}+n(n+k-2)\,s_{n-1,k}.
\]
Using Proposition~\ref{prop:s_coef} we get
\[
\frac{n}{2}\,s_{n-1,k+1}
=\frac{n}{2}\cdot\frac{(n-1)!(n+k-2)!}{2^{k}k!(n-k-2)!}
=\frac{n!(n+k-2)!}{2^{k+1}k!(n-k-2)!},
\]
and
\[
n(n+k-2)\,s_{n-1,k}
= n(n+k-2)\cdot\frac{(n-1)!(n+k-3)!}{2^{k-1}(k-1)!(n-k-1)!}
=\frac{n!(n+k-2)!}{2^{k-1}(k-1)!(n-k-1)!}.
\]
Factoring out $\frac{n!(n+k-2)!}{2^{k+1}k!(n-k-1)!}$ yields
\[
|\STC_{n,k}|
=\frac{n!(n+k-2)!}{2^{k+1}k!(n-k-1)!}\,\bigl((n-k-1)+4k\bigr)
=\frac{n!(n+k-2)!(n+3k-1)}{2^{k+1}k!(n-k-1)!},
\]
as claimed.
\end{proof}

\subsection{Counting non-labeled spinal tree-child networks}

    Proceeding as in the labeled case, we can define the generating function of $\NLMT$ as $D(x,z) = \sum_{n,k\geq 1} \frac{d_{n,k}}{k!}x^nz^k$.

\begin{prop}\label{prop:d_coef}
The coefficients $d_{n,k}$ satisfy the following explicit formulation:
   $$d_{n,k}=\left\{ \begin{array}{ccl}
   0 & if  & n<k \\
   \displaystyle\frac{(n+k-2)!}{2^{k-1}(k-1)!(n-k)!}& if  & n\geq k
\end{array}\right. $$
\end{prop}

\begin{proof}
The argument is analogous to the one used in Proposition~\ref{prop:s_coef}, which treated the labeled case.  The combinatorial decomposition of the objects is identical, and therefore the counting argument yields the same factorial expression. However, in this situation, we are working with a generating function that is ordinary with respect to $x$, since the structures are not labeled. Consequently, the coefficients $d_{n,k}$ in $D(x,z)$ are not divided by $n!$, unlike in the exponential generating function used in the labeled setting. Removing this factor we obtain the result.
\end{proof}

By direct combination of Propositions \ref{prop:marked_unlabeled} and \ref{prop:d_coef}, we obtain the explicit enumeration formula for the class \( \NLSTC_{n,k} \). 

\begin{prop}
    For $n > 1$ we have
$$|\mathcal{NLSTC}_{n,k} |=  \frac{(n+k-1)!}{2^kk!(n-k-1)!}.$$
\end{prop}

\section{Conclusions}\label{sec:conclusions}


In this work we gave two complementary combinatorial derivations for the enumeration of spinal tree-child networks: a word-encoding framework that yields an explicit closed count for the unlabeled class (and recovers the labeled formula via the labeled/unlabeled relation), and a symbolic approach based on marked graphs leading to a solvable bivariate generating function.

There are several natural generalizations of spinal tree-child networks that we plan to investigate in future work; for instance, allowing more general tree-shaped attachments off the main spine (subject to additional structural constraints) instead of restricting to single leaves.

\section*{Acknowledgments}
JCP and PV would like to thank the Institute for Computational and Experimental Research in
Mathematics (ICERM) in Providence, Rhode Island, for the welcoming environment that supported part of the writing of this article. PV, JCP and GC were supported in part from Grant PID2021-126114NB-C44 funded 
by MCIN/AEI/ 10.13039/501100011033.   The second author was supported by the Grant PID2023-147202NB-I00 funded by \\MICIU/AEI/10.13039/501100011033.

\section*{Conflict of interest}
On behalf of all authors, the corresponding author states that there is no conflict of interest.

\bibliographystyle{alpha}
\bibliography{references}











\end{document}